\documentclass[letterpaper,11pt,reqno]{amsart}

\makeatletter
\usepackage{amssymb}
\usepackage{latexsym}
\usepackage{amsbsy}
\usepackage{amsfonts}
\usepackage{graphicx}
\usepackage{enumerate}
\usepackage{mathtools}
\usepackage{color}
\usepackage{comment}

\usepackage{tikz}

\def\marginpar#1{\ignorespaces}

  \newcommand{\beq}{\begin{equation}}
    \newcommand{\eeq}{\end{equation}}
    
    \newcommand{\bal}{\begin{align}}
    \newcommand{\eal}{\end{align}}
    \newcommand{\bals}{\begin{align*}}
    \newcommand{\eals}{\end{align*}}
    
    \newcommand{\calA}{{\mathcal A}}
    \newcommand{\bbN}{{\mathbb{N}}}
    \newcommand{\bbR}{{\mathbb{R}}}

    \newcommand{\bbZ}{{\mathbb{Z}}}

    \newcommand{\bbS}{{\mathbb{S}}}

    \newcommand{\calL}{{\mathcal L}}
    
    \newcommand{\calF}{{\mathcal F}}
    \newcommand{\calH}{{\mathcal H}}

    \newcommand{\eps}{\varepsilon}
    
    \newcommand{\lb}{\label}

\DeclareMathOperator\Lip{Lip}

\DeclareMathOperator\argmin{\arg \min}

\newtheorem{theorem}{Theorem}[section]
\newtheorem{remark}{Remark}[section]
\newtheorem{lemma}[theorem]{Lemma}
\newtheorem{proposition}[theorem]{Proposition}

\numberwithin{equation}{section}
\makeatother
\begin{document}
\title[]{Policy iteration for the deterministic control problems -- a viscosity approach}

\author[W. Tang, H. V. Tran, Y. P. Zhang]{Wenpin Tang, Hung Vinh Tran, Yuming Paul Zhang}

\thanks{
The work of Tang is supported by NSF grant DMS-2206038, the ColumbiaInnovation Hub grant, and the Tang Family Assistant Professorship. 
The work of Tran partially supported by NSF CAREER grant DMS-1843320, a Simons Fellowship, and a VilasFaculty Early-Career Investigator Award. 
The work of Zhang is partially supported bySimons Foundation Travel Support MPS-TSM-00007305 and a start-up grant at Auburn University.
}

\address[W. Tang]{Department of Industrial Engineering and Operations Research, Columbia University, S.W. Mudd Building, 
500 W 120th St, New York, NY 10027} 
\email{wt2319@columbia.edu}

\address[H. V. Tran]
{
Department of Mathematics, 
University of Wisconsin Madison, Van Vleck Hall, 480 Lincoln Drive, Madison, WI 53706}
\email{hung@math.wisc.edu}

\address[Y. P. Zhang]
{
Department of Mathematics and Statistics, Auburn University, Parker Hall, 
221 Roosevelt Concourse, Auburn, AL 36849}
\email{yzhangpaul@auburn.edu}

\date{\today} 

\keywords{Finite differences, Hamilton-Jacobi-Bellman equations, optimal control, policy iteration}

\begin{abstract}
This paper is concerned with the convergence rate of policy iteration for (deterministic) optimal control problems 
in continuous time.
To overcome the problem of ill-posedness due to lack of regularity, 
we consider a semi-discrete scheme by adding a viscosity term via finite differences in space. 
We prove that the Policy Iteration (PI) for the semi-discrete scheme converges exponentially fast, and provide a bound on the error induced by the semi-discrete scheme.
We also consider the discrete space-time scheme, where both space and time are discretized. 
The convergence rate of PI and the discretization error are studied. 
\end{abstract}

\maketitle


\section{Introduction}
\label{sc1}

\quad Optimal control is ubiquitous in science and engineering with a variety of applications including
aerospace engineering \cite{Ben10, BAS10}, chemical engineering \cite{NV21}, economy \cite{KS91}, 
operations research \cite{RKP16, SZ94} and robotics \cite{BCHT17, CT18}.
Dynamic programming (DP) has proved to be an efficient tool for solving multistage optimal control problems 
since its inception by Bellman \cite{Bellman57}.
In recent years, reinforcement learning (RL) has shown great success in resolving complex decision-making problems,
notably AlphaGo \cite{Silver17} and humanoid tasks \cite{Haa18}.
Policy iteration (PI), as a class of approximate or adaptive dynamic programming (ADP), 
is instrumental in many RL algorithms \cite{SB18}.

\quad The idea of PI dates back to Howard \cite{Howard60} in a stochastic environment known as the Markov decision process (MDP).
Subsequent works \cite{Bert12, Powell07, Put94} explored PI for MDPs in discrete time and space;
recently, \cite{Ber17, LJMB21} considered PI for (deterministic) optimal control problems in discrete time and continuous space. 
In these works, PIs are proved to converge to the optimal control under suitable conditions on the model parameters. 
Furthermore, \cite{puterman1979, santos2004} studied the convergence rate of PI for infinite-horizon MDP.
On the other hand, many real-world problems are modeled by dynamical systems evolving in continuous time, 
and it is known that DP for optimal control in continuous time and space entails the Hamilton-Jacobi-Bellman (HJB) partial differential equation (PDE).
Despite its importance, PI for optimal control problems in continuous time and space 
has mostly been studied in the linear quadratic setting \cite{Klein68, VPAL09},
or those with a specific structure that allows solvability to some extent \cite{AL05}.
It was not until recently that the general space-time problems were considered in \cite{LS21} under a fixed point assumption.
For the stochastic control problems, \cite{KSS20, RSZ22} showed that PI converges exponentially fast in the case where
controls are only exercised on the drift term of the state process.
Similar results were derived for the corresponding entropy-regularized problems \cite{HWZ22, TZ23}.
Recently, PI for mean field games was considered in \cite{CCG21, C22, CT22}. 
We also mention, in a closely related direction, \cite{BJ16, WLL15} studied value iteration for optimal control problems. 
\cite{JM70, Mayne66} proposed differential dynamic programming.
It relies on a quadratic approximation to the value function, which requires the second-order property of the model parameters.
See \cite{LV09, WHL17} for recent progress on the theory and applications of ADP for optimal control and RL.

\quad In this paper, we study the convergence rate of PI for optimal control problems in continuous time and their discretization
under general conditions on the model parameters. We will assume that the cost function, the control, and the vector field that controls the system's state are all uniformly bounded and Lipschitz continuous. However, some of our results hold under more general assumptions (see Remark \ref{R.1}). 
Note that the convergence analysis in \cite{AL05, Klein68, VPAL09} relies on the specific structure of the problem, while \cite{LS21} assumed that the HJB operator enjoys a fixed point or a contraction property which is hard to verify. 
None of these works quantified the convergence of PI to the optimal control. 
Moreover, PI for continuous-time control problems may even be ill-posed due to lack of regularity.
Our idea is to introduce a viscosity term ``$h \Delta^h$" in the policy evaluation,
where $h$ is the mesh size and $\Delta^h$ is the discrete Laplacian in space. 
We call it a {\em semi-discrete scheme}. 
Essentially the viscosity term is of order $1$, which assures that the finite difference scheme is monotone. 
A monotone scheme is commonly desirable for numerical implementation
so the addition of the finite difference viscosity term is natural. 
On the other hand,
the viscosity term in the semi-discrete scheme 
mimics the vanishing viscosity approximation to first-order PDEs \cite{Evans80},
which forces PI to converge exponentially fast (Theorem \ref{T.2.4} and Theorem \ref{T.2.add}) as for the stochastic control problems.
We also prove that 
the discrepancy between the optimal control problem and its semi-discrete scheme is of order $\sqrt{h}$ as $h \to 0$ (Theorem \ref{T.2.5}).
If further assuming the cost function and the vector field to be uniformly bounded in $W^{2,\infty}$ in space, then the policy in PI converges almost everywhere  (Theorem \ref{T.2.8}). 
Further, we consider the time-discretization, called a {\em discrete space-time scheme}.
The same results hold for PI for the discrete space-time scheme (Theorem \ref{T.4.2} and Theorem \ref{T.4.3}).
Our results echo recent work \cite{HXY21},
which asserts that noise enhances the convergence of finite-horizon RL algorithms.
In our setting, noise corresponds to the viscosity term,
and the importance of finite horizon is seen from various bounds with exponential dependence in time. 
Our analysis relies on PDE techniques
(which is also useful in analyzing vanishing viscosity approximations for mean field games \cite{TZP23}),
and may carry over to the study of differential games in solving Hamilton-Jacobi-Bellman-Issacs (HJBI) equations.

\quad 
To the best of our knowledge, the exponential convergence results in Theorems  \ref{T.2.4},  \ref{T.2.add}, and \ref{T.4.2} are new in the literature and they are essentially optimal.
For the quantitative convergence of the solutions to the semi-discrete scheme and the discrete space-time scheme to these of the continuous equations in Theorems  \ref{T.2.5} and \ref{T.4.3}, we follow the approach of Crandall and Lions \cite{crandall1984two}.
Note that \cite{crandall1984two} does not deal with PI and approximated optimal policies.

\quad The rest of the paper is organized as follows. 
In Section \ref{sc2}, we provide background and present the semi-discrete and the discrete space-time schemes.
In Section \ref{sc3} we study the semi-discrete scheme, and in Section \ref{sc4} we analyze the discrete space-time scheme.
We provide further PDE perspectives in Section \ref{sc5}.
We conclude with Section \ref{sc6}.

\section{Setup and preliminary results}
\label{sc2}

\quad In this section, we present the semi-discrete and the discrete space-time schemes.
Consider a system whose state is governed by the ordinary differential equation:
\begin{equation}\lb{2.1}
\frac{d x(t)}{dt} = f(t, x(t), \alpha(t)),
\end{equation}
where for $0 \le t \le T$, $x(t) \in \mathbb{R}^d$ is the system state, and $\alpha(t) \in A \subset \mathbb{R}^{m}$ is the control or policy, and $f:[0,T]\times \bbR^d\times A\to\bbR^d$ is Lipschitz continuous. 
Here, $A$ is a given compact subset of $\mathbb{R}^m$.
The objective is 
\begin{equation}\lb{2.2}
J(t,x,\alpha):= \int_t^T c(s, x(s), \alpha(s))\,ds + q(x(T)) \quad \mbox{given } x(t) = x,
\end{equation}
and the goal is to minimize this objective function. 
Denote by
\begin{equation}
\label{eq:vcontrol}
v_*(t,x):= \inf_{\alpha \in \mathcal{A}_t} J(t,x,\alpha),
\end{equation}
where $\mathcal{A}_t$  is the standard admissible policy defined as
$\mathcal{A}_t= \{\alpha:[t,T] \to A \,:\, \alpha \text{ is measurable} \}$.
It is known that under suitable conditions on $c(\cdot)$ and  $q(\cdot)$ (see \cite[Chapter 2]{FS06} or \cite[Chapter 2]{tranbook}), 
$v_*$ defined by \eqref{eq:vcontrol} is the viscosity solution to  
\begin{equation}
\label{3.1}
\begin{cases}
\partial_t v(t,x) + H(t,x, \nabla v(t,x)) = 0 \qquad &\text{ in } (0,T)\times \mathbb{R}^d,\\
 v(T,x) = q(x) \qquad &\text{ on }  \mathbb{R}^d,
 \end{cases}
\end{equation}
where  the Hamiltonian $H:[0,T]\times\bbR^d\times\bbR^d\to \bbR$ is given by
\[
H(t,x,p):= \inf_{a \in A} \left[c(t,x,a) + p \cdot f(t,x,a)\right].
\]
We assume the above infimum is achieved at a unique $a\in A$.
Denote by
\begin{equation}
\label{eq:u}
\alpha(t,x,p): = \argmin_{a \in A} \left[ c(t,x,a) + p \cdot f(t,x,a)\right].
\end{equation}
The optimal policy is given by
\begin{equation}
\label{eq:ustar}
\alpha_{*}(t,x) = \alpha(t,x, \nabla v_{*}(t,x)).
\end{equation}

We impose the following assumptions.
\begin{itemize}
    \item[(A1)] $c(\cdot,\cdot,\cdot),f(\cdot,\cdot,\cdot),q(\cdot)$ are uniformly bounded and Lipschitz continuous in all of their dependencies. 
     \item[(A2)] $\alpha(\cdot,\cdot,\cdot)$, the unique solution to \eqref{eq:u},   is uniformly Lipschitz continuous in all of its dependencies on $[0,T]\times \bbR^d\times A$. 
\end{itemize}
Condition (A2) is restrictive, 
which is required to ensure the well-posedness and regularity properties of the PI algorithm. 
It is hard to relax this condition as the control $\alpha$ appears directly in PI.

\quad Policy iteration is an approximate dynamic programming, 
which alternates between policy evaluation to get the value function with the current control
and policy improvement to optimize the value function.
More precisely, for $n = 0,1, \cdots$, the iterative procedure is:
\begin{itemize}
\item
Given $\alpha_n(t,x)$, solve the linear PDE
\begin{equation}
\label{eq:PDEiter}
\begin{cases}
\partial_t v_n(t,x) + c(t,x,\alpha_n(t,x)) + \nabla v_n(t,x) \cdot f(t,x,\alpha_n(t,x)) = 0 \quad &\text{ in } (0,T)\times \mathbb{R}^d,\\
v_n(T,x) = q(x) \quad &\text{ on }  \mathbb{R}^d.
 \end{cases}
 \end{equation}
\item
Set
\begin{equation}
\label{eq:uiter}
\alpha_{n+1}(t,x) = \alpha(t,x, \nabla v_n(t,x)) = \argmin_{a\in A} \left[c(t,x,a) + \nabla v_n(t,x) \cdot f(t,x,a)\right].
\end{equation}
\end{itemize}
The key is to understand how the sequence $\{v_n\}$ approximates the optimal value $v_{*}$, and how $\{\alpha_n\}$ approximates the optimal policy $\alpha_*$.

\quad On the other hand, it is not clear whether the policy iteration scheme \eqref{eq:PDEiter}--\eqref{eq:uiter} is well-posed. 
Intuitively, to make sense of $\alpha_{n+1}$ we need $v_n$ to be Lipschitz continuous, for which we then need $\alpha_n$ to be Lipschitz. 
This in turn requires $\nabla v_{n-1}$ to be Lipschitz. 
After iterations, we need $v_0$ to be smooth, which is not generally true.  

\quad Throughout the paper, we denote by $\bbN$ as the set of all positive natural numbers, and $\bbZ$ as the set of all integers. For any $h>0$, we write $h\bbZ^d:=\{hz\,|\,z\in\bbZ^d\}$. Let $\bbR^d$ be the Euclidean space of dimension $d$ and $|\cdot|$ the Euclidean distance. For $R>0$, by $B_R$ we mean the ball in $\bbR^d$ of radius $R$ and centered at the origin. For a vector field $f:[0,T]\times \bbR^d\times A\to \bbR^d$, we denote its infinity norm by $\|f\|_\infty$. For a function $g:[0,T]\times \bbR^d\to \bbR$, the spatial gradient is denoted as $\nabla g(t,x)=\nabla_x\, g(t,x)$ and the partial derivative with respect to time is denoted as $\partial_t g(t,x)$.

\quad We write $C$ as various universal constants that only depend on $d$ and the constants in (A1)--(A2) unless otherwise stated. Specifically, since $T,h$ are not universal constants, we keep track of the dependence on $T,h$ in most estimates. 
The constants $C$ might vary from one line to another. By $C_X$ or $C(X)$ we mean a constant that depends on universal constants and $X$. 

\subsection{Semi-discrete schemes}

For $T>0$, $h\in (0,1)$, $N\geq \max\{1,\|f\|_\infty/2\}$ and a given Lipschitz continuous function $\alpha_0:\bbR\times\bbR^{d}\to A$, we solve for $n=0,1,\cdots$
\beq\lb{p.1}
\left\{
\begin{aligned}
&\partial_t v_n^h(t,x) + c(t,x,\alpha_n(t,x)) + \nabla^h v_n^h(t,x) \cdot f(t,x,\alpha_n(t,x)) = - Nh\Delta^h v_n^h(t,x) &&\text{ in }(0,T)\times \bbR^d\\
&v_n^h(T,x) = q(x)  &&\text{ on }\bbR^d.
\end{aligned}
\right.
\eeq
Then set
\beq\lb{p.1'}
\alpha_{n+1}(t,x) = \alpha(t,x, \nabla^h v_n^h(t,x))\qquad \text{ in }(0,T)\times \bbR^d.    
\eeq
Here, for any $\varphi:\bbR^d\to\bbR$ and $h\in\bbR \setminus \{0\}$, we use the notations
\[
\nabla^h \varphi(x):=\left(\frac{\varphi(x+he_1)-\varphi(x-he_1)}{2h},\cdots,\frac{\varphi(x+he_d)-\varphi(x-he_d)}{2h}\right),
\]
\[
\Delta^h \varphi(x):=\sum_{i=1}^d\frac{\varphi(x+he_i)-2\varphi(x)+\varphi(x-he_i)}{h^2}.
\]
Later we will also write
$D^h \varphi(x):=\left(\frac{\varphi(x+he_1)-\varphi(x)}{h},\cdots,\frac{\varphi(x+he_d)-\varphi(x)}{h}\right)$.
It is clear that 
\beq\lb{p.6}
\nabla^h \varphi(x)=\frac{1}{2}\left(D^h \varphi(x)+D^{-h} \varphi(x)\right).
\eeq

\quad The assumption $N\geq \|f\|_\infty/2$ guarantees that the numerical Hamiltonian is monotone and, as a consequence of this, the following comparison principle holds (see e.g., \cite{crandall1984two,osher1991,tranbook}). 

\begin{lemma}\lb{L.2.1}
Let $v_0^h$ and $\tilde{v}_0^h$ be, respectively, a bounded continuous super- and sub- solution to \eqref{p.1} with $n=0$, and satisfy $\tilde{v}_0^h\leq {v}_0^h$ at $t=T$. 
Then $\tilde{v}_0^h\leq {v}_0^h$ in $[0,T]\times\bbR^d$.  Here by a supersolution (resp. subsolution), we mean that it satisfies \eqref{p.1} with the first equality replaced by $\leq$ (resp. $\geq$) and the second equality replaced by $\geq$ (resp. $\leq$).
\end{lemma}

\quad First, we show that the scheme \eqref{p.1}--\eqref{p.1'} is well-posed.

\begin{proposition}\lb{P.2.0}
Assume {\rm (A1)--(A2)} and $N\geq \max\{1,\|f\|_\infty/2\}$. Then the iterative process \eqref{p.1}--\eqref{p.1'} is well-defined, that is, there are Lipschitz continuous functions $v_n^h,\alpha_n$ satisfying \eqref{p.1}--\eqref{p.1'} and $v_n^h$ are uniformly bounded for all $n\geq 0$ and $h>0$. 
\end{proposition}

\begin{proof}
Since $\alpha_0$ is Lipschitz continuous, the unique solvability of \eqref{p.1} for $n=0$ follows from \cite[Theorem 2.4]{18Kry}. 
If one can show that $v_0^h$ is uniformly bounded and Lipschitz continuous with Lipschitz constant $C_h$, then $\alpha_1$ is Lipschitz continuous with Lipschitz constant $C_h'/h$ for some $C_h'>0$ by the assumption that $\alpha$ is Lipschitz and
\begin{align*}
&\left|\nabla^h v_0^h(t,x)-\nabla^h v_0^h(s,y)\right|\\
\leq\ & \left|\Big(\frac{v_0^h(t,x+he_1)-v_0^h(s,y+he_1)-v_0^h(t,x-he_1)+v_0^h(s,y-he_1)}{2h},\cdots\Big)\right|\\
\leq \ &  h^{-1}C_h(|x-y|+|t-s|). 
\end{align*}
From the same argument, we obtain a unique bounded and Lipschitz solution $v_1^h$. 
The existence of solutions then follows from iterations.

\quad First we prove the boundedness of $v_0^h$. 
Since $c(\cdot,\cdot,\cdot),q(\cdot)$ are uniformly bounded, 
we have that $\pm \left[\|q\|_\infty+\|c\|_\infty(T-t)\right]$
are a supersolution and a subsolution to \eqref{p.1} with $n=0$, respectively.
Hence, by Lemma \ref{L.2.1},
\[
-\|q\|_\infty-\|c\|_\infty(T-t)\leq v_0^h(t,x)\leq \|q\|_\infty+\|c\|_\infty(T-t),
\]
for all $(t,x)\in [0,T]\times \bbR^d$. 
The same bound holds for all $v_n^h$ by this argument.

\quad Next we show that $v_0^h$ is Lipschitz continuous with Lipschitz constant independent of $h$ when $T=T_0$ is sufficiently small depending only on the Lipschitz norms of $c,f$ and $\alpha_0$ presented in assumptions (A1)--(A2). 
The general result for any $T>0$ follows immediately by iterations and shifting in time on $[0,T_0]$, $[T_0,2T_0]$, \ldots, to $[k T_0, (k+1)T_0]$ where $kT_0 < T \leq (k+1)T_0$ for some $k\in \mathbb N$. 
For simplicity of notation, write 
\[
G(t,x,p):=c(t,x,\alpha_0(t,x))+p\cdot f(t,x,\alpha_0(t,x)).
\]
Then for $M:=2\|\nabla q\|_\infty+1$, define
\[
\tilde{G}(t,x,p):=\left\{
\begin{aligned}
&G(t,x,p)   \qquad &&\text{ if }|p|\leq M,\\
&G(t,x,Mp/|p|) \qquad &&\text{ if }
|p|> M.
\end{aligned}
\right.
\]
It follows from (A1) and the Lipschitz continuity of $\alpha_0$ that $G$ is Lipschitz continuous in $(t,x)$ with Lipschitz constant $C(1+|p|)$. Thus, also using that $\|f\|_\infty\leq 2N$, we get for all $t,x,p$,
\beq\lb{p.21}
|\tilde G_t(t,x,p)|,\, | \tilde G_x(t,x,p)|\leq C(1+M),\quad | \tilde G_p(t,x,p)|\leq 2N
\eeq
where $C$ only depends on the Lipschitz norms of $c,f$ and $\alpha_0$.

\quad Now let $\tilde v^h$ be the solution to
\[
\left\{
\begin{aligned}
&\partial_t \tilde v^h(t,x) + \tilde G(t,x,\nabla^h \tilde v^h(t,x)) = - Nh\Delta^h \tilde v^h(t,x),\\
&\tilde v^h(T,x) = q(x).
\end{aligned}
\right.
\]
The goal is to show that $\tilde{v}^h$ is Lipschitz continuous, and $\tilde{v}^h={v}^h_0$ in $[0,T]\times \bbR^d$.

\quad It follows from the equation of $\tilde{v}^h$ that
$
p_s(t,x):=\frac{\tilde v^h(t,x+se)-\tilde v^h(t,x)}{s}
$  for any $e\in\bbS^{d-1}$ and $s\in(0,1)$
satisfies
\begin{equation}\label{eq:ps}
\left\{
\begin{aligned}
&\partial_t p_s(t,x) + G_1(t,x)+G_2(t,x)\cdot \nabla^h p_s(t,x) = - Nh\Delta^h p_s(t,x) \quad &&\text{ in }(0,T)\times \bbR^d,\\
&p_s(T,x) = \frac{q(x+se)-q(x)}{s} \quad &&\text{ on }\bbR^d,
\end{aligned}
\right.
\end{equation}
where
\begin{align*}
    G_1(t,x)&:=\frac1s\int_0^{s}\tilde G_x\left(t,x+ze,\nabla^h \tilde v^h(t,x+se)\right)\cdot e \,dz,\\
    G_2(t,x)&:=\int_0^1 \tilde G_p\left(t,x, \nabla^h \tilde v^h(t,x)+z(\nabla^h \tilde v^h(t,x+se)-\nabla^h \tilde v^h(t,x))\right)\,dz.
\end{align*}
It is clear from \eqref{p.21} that 
$|G_1|\leq C(1+M)$ and $|G_2|\leq 2N$.
This yields the comparison principle for \eqref{eq:ps} holds.
Thus, by comparing $p_s$ with $\pm (\|\nabla q\|_\infty+C(1+M)(T-t))$, we obtain
$|p_s(t,x)|\leq \|\nabla q\|_\infty+C(1+M)(T-t)$.
Sending $s\to 0$ yields for some $C$ depending only on (A1),
$|\nabla_e \tilde{v}^h(t,x)|\leq \|\nabla q\|_\infty+C(1+M)(T-t)$.
Thus, if $t\leq T\leq (2C)^{-1}$, we have that $\tilde v^h(t,x)$ is Lipschitz and 
\[
\sup_{(t,x)\in [0,T]\times\bbR^d}|\nabla\tilde v^h(t,x)|\leq \|\nabla q\|_\infty+1/2+M/2= M.
\]
From the definition of $\nabla^h$, we get
$
\sup_{(t,x)\in [0,T]\times\bbR^d}|\nabla^h \tilde v^h(t,x)|\leq  M.
$ Hence, $\tilde v^h$ is a solution to \eqref{p.1} for $n=0$. The uniqueness of the solution to \eqref{p.1} yields that $v^h_0\equiv \tilde v^h$. 
So we obtain the uniform Lipschitz continuity of $v_0^h$ in space, with the Lipschitz constant of the form $C\exp(CT)$. 
The Lipschitz regularity in time follows from the equation.
\end{proof}

\quad We point out that 
the Lipschitz constant of $v_n^h$ may depend on both $n$ and $h$ for $n\geq 1$. 
Another consequence of the comparison principle is that the functions $v_n^h$ are monotone decreasing in $n$.

\begin{proposition}\lb{P.2.1}
Under the assumptions of Proposition \ref{P.2.0}, we have for all $n\geq 0$, 
\[
 v_{n+1}^h(t,x)\leq v_n^h(t,x) \qquad \text{ for all $(t,x)\in [0,T]\times\bbR^d$.}
 \]
\end{proposition}
\begin{proof}
By the definition of $\alpha_n$,
\begin{align*}
&c(t,x,\alpha_{n+1}(t,x))   +\nabla^h v_{n}^h(t,x)\cdot f(t,x,\alpha_{n+1}(t,x))\\
&\qquad\qquad\leq c(t,x,\alpha_n(t,x))   +\nabla^h v_{n}^h(t,x)\cdot f(t,x,\alpha_n(t,x)). 
\end{align*}
Thus $v_n^h=v_n^h(t,x)$ is a supersolution to \eqref{p.1} with subscripts $n+1$ as it satisfies 
\[
\partial_t v_n^h + c(t,x,\alpha_{n+1}(t,x)) + \nabla^h v_n^h \cdot f(t,x,\alpha_{n+1}(t,x)) \leq - Nh\Delta^h v_n^h \quad \text{ in }(0,T)\times \bbR^d.
\]
Therefore, the comparison principle (Lemma \ref{L.2.1}) yields $v^h_{n+1}\leq v^h_{n}$ in $[0,T]\times\bbR^d$ for each $n\geq 0$.
\end{proof}

\quad Since $v_n^h$ is uniformly bounded for all $n\geq 0$, the monotonicity property of $v_n^h$ in $n$ from Proposition \ref{P.2.1} yields that $v_n^h$ converges locally uniformly as $n\to\infty$. 
We denote the limit as $v^h$. 
Then by the stability property of viscosity solutions, $v^h$ solves
\beq\lb{p.2}
\left\{
\begin{aligned}
&\partial_t v^h(t,x) + H(t,x,\nabla^h v^h(t,x)) = - Nh\Delta^h v^h(t,x) \quad &&\text{ in }(0,T)\times \bbR^d,\\
&v^h(T,x) = q(x) \quad &&\text{ on } \bbR^d,
\end{aligned}
\right.
\eeq
where
\beq\lb{p.15}
\begin{aligned}
H(t,x,p)&:=c(t,x,\alpha(t,x,p)) + p \cdot f(t,x,\alpha(t,x,p))\\
&=\min_{a \in A}\left[ c(t,x,a)+p\cdot f(t,x,a)\right].
\end{aligned}
\eeq
Since $\alpha(t,x,p)$ is assumed to be uniformly Lipschitz continuous in all of its dependencies, there exists $C>0$ such that for all $(t,x,p)\in [0,T]\times\bbR^d\times\bbR^d$,
\beq\lb{2.5}
|H_t(t,x,p)|,\, |H_x(t,x,p)|\leq C(1+|p|),\quad |H_p(t,x,p)|\leq C.
\eeq
The same proof of uniform boundedness and Lipschitz continuity for $v_0^h$ in Proposition \ref{P.2.0} can show that
$v^h$ is uniformly bounded and Lipschitz continuous, and the estimates are uniform in $h>0$. 
Below we also consider the unique solution $v$ to \eqref{3.1}, as one expects that it is the limit of $v^h$ as $h\to 0$. We will prove this fact in Theorem \ref{T.2.5}. 

\begin{lemma}\lb{L.2.3}
Under the assumptions of Proposition \ref{P.2.0}, let $v^h_0,v^h$ and $v$ be, respectively, solutions to \eqref{p.1} (for $n=0$), \eqref{p.2} and \eqref{3.1}. Then in $[0,T]\times\bbR^d$, $v_0^h,v^h$ and $v$ are bounded by $C(1+T)$ and are Lipschitz continuous with Lipschitz constant $C\exp(CT)$ for some universal constant $C>0$.

    
    
\end{lemma}


\quad For a general class of first-order Hamilton-Jacobi (continuous) equations, we refer to 
\cite{barles1990regularity,barles} for the regularity results.

\subsection{Discrete space-time schemes}

Now we consider the scheme that is discrete in both space and time.
Let ${\tau,h}\in (0,1)$, and $N$ such that
\beq\lb{N}
\max\{1,\|f\|_\infty/2\}\leq N\leq h/(2d\tau). 
\eeq
Assuming that $T/\tau\in\bbN$, we denote 
\[
\bbN_{T}^\tau:=\{0,\tau,2\tau,\cdots,T\},\quad \bbZ^d_h:=h\bbZ^d,
\]
\[
\Omega^{\tau,h}_{T}:=\bbN_{T}^\tau\times \bbZ^d_h\quad\text{and}\quad \Omega_{T}':=(\bbN_{T}^\tau\backslash\{0\})\times \bbZ^d_h.
\]
Given a Lipschitz continuous function $\alpha_0(t,x)$, let $V_n^{\tau,h}: \Omega^{\tau,h}_{T}\to \bbR$ be defined iteratively for $n=0,1,\cdots$ as follows: 
\beq\lb{4.1}
\left\{
\begin{aligned}
&\partial_t^\tau V_n^{\tau,h}(t,x) + c(t,x,\alpha_n(t,x)) + \nabla^h V_n^{\tau,h}(t,x) \cdot f(t,x,\alpha_n(t,x)) = - Nh\Delta^h V_n^{\tau,h}(t,x) &&\text{ in } \Omega_{T}',\\
&V_n^{\tau,h}(T,x) = q(x) &&\text{ on }\bbZ^d_h
\end{aligned}
\right.
\eeq
with
\beq\lb{4.1'}
\alpha_{n+1}(t,x): = \alpha(t,x, \nabla^h V_n^{\tau,h}(t,x))\qquad \text{ in }\Omega_{T}'.    
\eeq
Here we used the notation
$\partial_t^\tau V_n^{\tau,h}(t,x):=\frac{V_n^{\tau,h}(t,x)-V_n^{\tau,h}(t-\tau ,x)}\tau$. 

\quad We also consider the following equation
\beq\lb{4.2}
\left\{
\begin{aligned}
&\partial_t^\tau V^{\tau,h}(t,x) + H(t,x,\nabla^h V^{\tau,h}(t,x)) = - Nh\Delta^h V^{\tau,h}(t,x)\qquad &&\text{ in } \Omega_{T}',\\
&V^{\tau,h}(T,x) = q(x)\qquad &&\text{ on }\bbZ^d_h
\end{aligned}
\right.
\eeq
where 
$H$ is given by \eqref{p.15}. 
The goal is to show that $V_n^{\tau,h}$ converges to $V^{\tau,h}$ as $n\to\infty$, and $V^{\tau,h}$ converges to $v$ as $\tau,h\to 0$, where $v$ is given by \eqref{3.1}.

\quad We will use the following operator. For each $t\in 
\bbN_{T}^\tau$, let $\calF_t: L^\infty(\bbZ_h^d)\to L^\infty(\bbZ_h^d)$ be defined as
\beq\lb{calF}
\calF_t(U)(x):=U(x)+\tau H(t,x,\nabla^h U(x))+Nh\tau \Delta^h U(x).
\eeq
Then the equation in \eqref{4.2} can be rewritten as
$V_n^{\tau,h}(t-\tau,x)=\calF_t(V_n^{\tau,h}(t,\cdot))(x)$.
We need 
\[
\max\{1,\|H_p\|_\infty/2\}\leq N\leq h/(2d\tau),
\]
(which corresponds to \eqref{N} as $\|H_p\|_\infty=\|f\|_\infty$) to guarantee a monotonicity property of the operator $\calF_t$.
That is, for all $t\in\bbN_{T}^\tau$ and $U,V\in L^\infty(\bbZ_h^d)$ satisfying $U\leq V$, we have $\calF_t(U)\leq \calF_t(V)$, see e.g., \cite{crandall1984two,tranbook}. It is easy to see that the same holds if we replace $H(t,x,p)$ by $c(t,x,\alpha_n(t,x))+p\cdot f(t,x,\alpha_n(t,x))$ as $\|f\|_\infty\leq 2N$.

\quad The monotonicity property is important because it immediately implies the comparison principle of \eqref{4.2} and the scheme \eqref{4.1}--\eqref{4.1'}, in the sense that is similar to Lemma \ref{L.2.1}. 
As a consequence of this, one can show the following properties.


\begin{proposition}\lb{P.4.1}
Assume {\rm (A1)--(A2)} and \eqref{N}. Then  in $\Omega^{\tau,h}_{T}$, the solutions $V_n^{\tau,h},V^{\tau,h}$ are bounded by $C(1+T)$, and are Lipschitz continuous with Lipschitz constant $C\exp(CT)$  for some universal constant $C>0$. Moreover for all $n\geq 0$ we have $ V_{n+1}^{\tau,h}(t,x)\leq V_n^{\tau,h}(t,x)$ for all $(t,x)\in \Omega^{\tau,h}_{T}$. 
\end{proposition}

\quad The proof of Proposition \ref{P.4.1} is similar to those of Propositions \ref{P.2.0}, \ref{P.2.1} and Lemma \ref{L.2.3},
and hence we skip it.

\section{Analysis of semi-discrete schemes}
\label{sc3}

\subsection{Convergence of PI}
We show that for each fixed $h\in(0,1)$, $v^h_n\to v^h$ as $n\to\infty$ exponentially fast in $L^2_{loc}$ norm. We will assume $T\geq 1$ for convenience.

\begin{theorem}\lb{T.2.4}
Assume {\rm (A1)--(A2)} and $N\geq  1$. 
Let $v^{h}_n$ and $v^{h}$ be, respectively, continuous solutions to \eqref{p.1} and \eqref{p.2}. 
Then there exists a universal constant $C>0$ such that for all $n\geq 1$, $R\geq 1$ and $t \in [0,T]$ we have
\[
\begin{aligned}
&\int_{B_R}\left|v^h_n(t,x)-v^h(t,x)\right|^2\,dx \leq  \frac{h}{2^{n+1}} \int_t^T \int_{\bbR^d}\exp\left[C(1+\|\nabla^h v^h\|_\infty^2)(s-t)/h\right]\times
 \\
&\qquad\qquad\qquad \left|D^h(v^h_0(s,x)-v^h(s,x))\right|^2\min\left\{1,e^{-|x|+R+1}\right\}\, dxds.
\end{aligned}
\]
In particular, we have
$\sup_{t \in [0,T]}\int_{B_R}\left|v^h_n(t,x)-v^h(t,x)\right|^2\,dx \leq C2^{-n}\exp\left[C\exp(CT)/h\right]R^d$.
\end{theorem}


\begin{proof}
In this proof, we write $v_n:=v^{h}_n$ and $v:=v^{h}$, and assume $T\geq 1$ for simplicity.
For any fixed $R\geq 1$, let $\varphi=\varphi_R:[0,\infty)\to (0,1]$ be $C^1$ and satisfying
\beq\lb{p.8}
\begin{aligned}
&\varphi(r)=1\quad\text{ on }[0,R],\qquad \varphi(r)=e^{-r+R}\quad\text{ on $[R+1,\infty)$},\\
&-\varphi'(r)\in [0,4\varphi(r)]\quad \text{ for all $r>0$}.
\end{aligned}
\eeq
It is clear that such $\varphi$ exists. Later we write $\varphi(x):=\varphi(|x|)$ for $x\in\bbR^d$.

\quad Next, for some $A>0$ to be determined, set 
\beq\lb{p.7}
E_{t,n}:=\frac{1}{2}e^{At}\int_{\bbR^d}|v_n(t,x)-v(t,x)|^2\varphi(x)\,dx
\eeq
which is finite since $v_n,v$ are uniformly bounded. 
Direct computation yields
\beq\lb{p.4}
\frac{d}{dt}E_{t,n}=AE_{t,n}+e^{At}\underbrace{\int_{\bbR^d}(v_n(t,x)-v(t,x))\left(\partial_tv_n(t,x)-\partial_tv(t,x)\right)\varphi(x)\,dx}_{=:X_{t,n}}.
\eeq

\quad Recall from \eqref{p.15} that
$H(t,x,p)=c(t,x,\alpha(t,x,p))+\nabla v\cdot f(t,x,\alpha(t,x,p))$.
Below, we write 
\begin{align*}
c:=c(t,x,\alpha(t,x,\nabla^h V^{\tau,h}(t,x)))\quad &\text{and}\quad f:=f(t,x,\alpha(t,x,\nabla^h V^{\tau,h}(t,x)))\\
c_n:=c(t,x,\alpha_n(t,x)) \quad &\text{and}\quad f_n:=f(t,x,\alpha_n(t,x)))
\end{align*}
for simplicity. We will also drop $(t,x)$ from the notations of $v(t,x)$ and $v_n(t,x)$, and $(x)$ from $\varphi(x)$ when there is no confusion.
Direct computation yields
\begin{align*}
    \int_{\bbR^d}(\Delta^h v)v\varphi\, dx&=-\int_{\bbR^d}|D^h v|^2\varphi\, dx+\frac{1}{h^2}\sum_{i=1}^d\int_{\bbR^d}v(t,x+he_i)(v(t,x+he_i)-v(t,x))\varphi(x)\, dx\\
&\qquad  \quad  -\frac{1}{h^2}\sum_{i=1}^d\int_{\bbR^d}v(t,x)(v(t,x)-v(t,x-he_i))\varphi(x)\, dx\\
&=-\int_{\bbR^d}|D^h v|^2\varphi\, dx-\int_{\bbR^d}v D^{-h} v\cdot D^{-h}\varphi\, dx,
\end{align*}
where the last equality was obtained by a change of variable. 
We then deduce from the equation that
\beq\lb{p.3}
\begin{aligned}
X_{t,n}&=-\int_{\bbR^d}(v_n-v)(\nabla^hv_n\cdot f_n+c_n+N h\Delta^hv_n-\nabla^hv\cdot f-c-Nh\Delta^hv)\varphi\, dx\\
&\geq  Nh\int_{\bbR^d}|D^h(v_n-v)|^2\varphi \,dx-Nh\int_{\bbR^d}|v_n-v||D^{-h}(v_n-v)||D^{-h}\varphi|\, dx\\
&\qquad\qquad-\int_{\bbR^d} |v_n-v|\left(|\nabla^h(v_n-v)||f_n|+|f_n-f||\nabla^h v|+|c_n-c|\right)\varphi\, dx.
\end{aligned}
\eeq

\quad Due to \eqref{p.8}, $|D^{-h}\varphi(x)|\leq C\varphi(x)$ for some constant $C>0$. 
Also, using $\|f\|_\infty<\infty$ and \eqref{p.6}, we have
$|\nabla^h(v_n-v)||f_n|\leq C(|D^h(v_n-v)|+|D^{-h}(v_n-v)|)$.
Since $v$ is Lipschitz continuous, $|\nabla^h v|\leq M$ for some $M\geq 1$.
So, by \eqref{p.1'} and the uniform Lipschitz continuity of $f,c$ and $\alpha$, we have for some $C>0$,
\beq\lb{p.10}
|f_n-f||\nabla^h v|+|c_n-c|
\leq CM(|D^h(v_{n-1}-v)|+|D^{-h}(v_{n-1}-v)|).
\eeq
With all these, if denoting
\[
G_{t,n}^h:=\int_{\bbR^d}|D^h(v_n(t,x)-v(t,x))|^2\varphi(x)\,dx,
\]
it follows from \eqref{p.3} that for some $C>0$,
\begin{align*}
   X_{t,n} &\geq  Nh G_{t,n}^h-C\int_{\bbR^d} |v_n-v|(|D^h(v_{n}-v)|+|D^{-h}(v_{n}-v)|)\varphi\, dx\\
&\qquad\qquad-CM\int_{\bbR^d} |v_n-v|(|D^h(v_{n-1}-v)|+|D^{-h}(v_{n-1}-v)|)\varphi\, dx.
\end{align*}

\quad Denote $w_n(t,x):=v_n(t,x)-v(t,x)$. Since $\varphi(x-he_i)\leq (1+Ch)\varphi(x)$ by the choice of $\varphi$, there exists $C>0$ such that 
\beq\lb{p.13}
\begin{aligned}
G_{t,n}^{-h} &=\int_{\bbR^d}\sum_{i=1}^d h^{-2}|w_n(t,x)-w_n(t,x-he_i)|^2\varphi(x)\,dx\\
&\leq (1+Ch)\int_{\bbR^d}\sum_{i=1}^d h^{-2}|w_n(t,x)-w_n(t,x-he_i)|^2\varphi(x-he_i)\,dx= (1+Ch)G_{t,n}^{h}.    
\end{aligned}
\eeq  
Then, using \eqref{p.7} and Young's inequality, we get for some universal $C>0$ and any $\sigma_1,\sigma_2 >0$,
\begin{align*}
X_{t,n}&\geq Nh G_{t,n}^h-\frac{\sigma_1}{2+Ch}\int_{\bbR^d} (|D^h(v_{n}-v)|^2+|D^{-h}(v_{n}-v)|^2)\varphi\, dx 
\\
&\qquad\quad -\frac{\sigma_2}{2+Ch}\int_{\bbR^d} (|D^{h}(v_{n-1}-v)|^2+|D^{-h}(v_{n-1}-v)|^2)\varphi\, dx \\
&\qquad\quad-C(2+Ch)(\sigma_1^{-1}+M^2\sigma_2^{-1})\int_{\bbR^d} |v_n-v|^2\varphi\, dx\\
&\geq  (Nh-\sigma_1)G_{t,n}^h-\sigma_2G_{t,n-1}^h-C(\sigma_1^{-1}+M^2\sigma_2^{-1})e^{-At}E_{t,n}.
\end{align*}
Using this and $E_{T,n}=0$, and integrating \eqref{p.4} over $[t,T]$, we obtain for some universal $C>0$,
\beq\lb{p.5}
\begin{aligned}
-E_{t,n}&\geq (A-C\sigma_1^{-1}-CM^2\sigma_2^{-1})\int_t^T E_{s,n}\,ds\\
&\qquad\quad+( Nh-\sigma_1)\int_t^Te^{As}G_{s,n}^h\,ds-\sigma_2\int_t^Te^{As}G_{s,n-1}^h\,ds.
\end{aligned}
\eeq

\quad Now taking 
$\sigma_1:={ h}/{2}$, $\sigma_2:={  h}/{4}$ and $A:=6CM^2/h$,
then \eqref{p.5} and $N\geq1$ yield
\[
\int_t^Te^{As}G_{s,n}^h\,ds\leq \frac12\int_t^Te^{As}G_{s,n-1}^h\,ds\leq\ldots\leq 2^{-n}\int_t^Te^{As}G_{s,0}^h\,ds.
\]
With this, \eqref{p.5} also shows that
$
E_{t,n}\leq \frac{  h}{4}\int_t^Te^{As}G_{s,n-1}^h\,ds\leq  \frac{h}{2^{n+1}}\int_t^Te^{As}G_{s,0}^h\,ds$.
Therefore, for all $n\geq 0$ and $t\in [0,T]$, we obtain 
\begin{align*}
\int_{B_R}|v_n(t,x)-v(t,x)|^2\,dx \leq 
\frac{h}{2^{n+1}} \int_t^T \int_{\bbR^d}e^{A(s-t)}|D^h(v_0(s,x)-v(s,x))|^2\varphi(x)\, dxds,
\end{align*}
which, combined with Lemma \ref{L.2.3}, concludes the proof.
\end{proof}

\begin{remark}
In the proof of Theorem \ref{T.2.4}, we only used the following:
uniform Lipschitz continuity of $f,c$ and $\alpha$, and uniform boundedness of $f$ and $|\nabla^h v^h|$. 
In particular, the solutions $v_n^h$ and $v^h$ are allowed to have certain growth at $x=\infty$, and the comparison principle is not needed. 
\end{remark}

\quad By Theorem \ref{T.2.4}, we immediately have the convergence of the policies.

\begin{theorem}\lb{T.2.add}
Assume {\rm (A1)--(A2)} and $N\geq  1$. 
Then there exists a universal constant $C>0$ such that for all $n\geq 0$ and $R\geq 1$ we have
\[
\sup_{t \in [0,T]}\int_{B_R}\left|\alpha(t,x,\nabla^h v^h_n(t,x))- \alpha(t,x,\nabla^h v^h(t,x))\right|^2\,dx\leq C2^{-n}\exp\left[C\exp(CT)/h\right]R^d.
\]
\end{theorem}
\begin{proof}
Since $\alpha$ is Lipschitz continuous,
\[
\begin{aligned}
&\int_{B_R}\left|\alpha(t,x,\nabla^h v^h_n(t,x))- \alpha(t,x,\nabla^h v^h(t,x))\right|^2\,dx \\ 
& \qquad  \leq  \frac{C}{h^2}\sum_{i=1}^d
\int_{B_R}\left|v^h_n(t,x+he_i)-v^h(t,x+he_i)-v^h_n(t,x-he_i)+v^h(t,x-he_i)\right|^2\,dx.
\end{aligned}
\]
We can then conclude the proof from Theorem  \ref{T.2.4}.
\end{proof}

\begin{remark}\lb{R.1}
Here we consider the problem where $f$ is linear in $x$, $c$, and $q$ are quadratic in $x$ (again we assume that the control set $A$ is compact). To compute the values and the optimal policy on $[0,T] \times B_R$ (then $(t,x)\in [0,T]\times B_R$), by \eqref{2.1} we have $|x(t)|\leq  CR e^{CT}$ for some $C>0$. Thus, by \eqref{2.2} and \eqref{eq:vcontrol}, we only need the information of $c,f$ and $q$ (and hence, $H$) for $|x| \leq C'R e^{C'T}$ for some $C'>0$. We can then perform a cut-off of $c,f$ and $q$ for $|x|\geq 2C'R e^{C'T}$, so that  $c,f$ and $q$ are globally bounded, and the value function and the optimal policy remain the same on $[0,T] \times B_R$. This shows that the boundedness conditions we impose are not restrictive.


\quad Of course, this argument does not work if we need to study the problem globally, but Theorems \ref{T.2.4} and \ref{T.2.add} deal exactly with this bounded setting.
\end{remark}

\subsection{Convergence of $v^h$ as $h\to 0$}
Let $v^h$ and $v$ be, respectively, solutions to \eqref{p.2} and \eqref{3.1}.
We show $|v^h- v|\leq C_T\sqrt{h}$, where the rate is sharp (we refer to a simple example given in \cite{10DonKry}).  
We also point out that for a semi-Lagrangian scheme (which preserves the optimization structure), it is possible to obtain a first-order estimate $O(h)$ if the discretized solution is semi-concave, see \cite{dolcetta1984,saluzzi2022}. However, our scheme is based on finite difference, and it is unclear whether or not $v^h$ is semi-concave. Along this line, our Theorem \ref{T.3.5} provides a weak semi-concavity result for $v^h$.

\begin{theorem}\lb{T.2.5}
Assume {\rm (A1)--(A2)} and $N\geq \max\{1,\|f\|_\infty/2\}$. 
Then there exists a universal constant $C>0$ such that
\[
\sup_{(t,x)\in [0,T]\times \bbR^d} |v(t,x)-v^h(t,x)|\leq C(1+T)(1+\|\nabla v\|_\infty) \sqrt{h}.
\]
In particular, we have
$\sup_{(t,x)\in [0,T]\times \bbR^d} |v(t,x)-v^h(t,x)|\leq C\exp(CT) \sqrt{h}$.
\end{theorem}

\begin{remark}
This rate was obtained in \cite{DonKry,18Kry}  for a large class of parabolic Bellman equations with Lipschitz coefficients. 
We apply a different argument -- the classical doubling variable method that is used in \cite{crandall1984two} in which a discrete space-time homogeneous Hamilton-Jacobi equation is discussed. 
This argument allows us to obtain the same sharp estimate for the scheme \eqref{4.1}, while it seems that the method in \cite{DonKry,18Kry} cannot (see Remark \ref{R.4}).
See also \cite{CGT-11} for a different proof of this convergence rate via the nonlinear adjoint method.
\end{remark}

\begin{proof}
We assume that $T\geq 1$. Suppose for some $(t_0,x_0)\in [0,T]\times \bbR^d$ such that
\beq\lb{3.2}
8\sigma:=v(t_0,x_0)-v^h(t_0,x_0)\geq \frac12 \sup_{(t,x)\in [0,T]\times \bbR^d}\left[ v(t,x)-v^h(t,x)\right]>0.
\eeq
Below we will show $\sigma\leq CT(1+\|\nabla v\|_\infty)\sqrt{h}$.

\quad Consider a smooth function $g:\bbR^{d+1}\to [0,1]$ such that
\begin{enumerate}
    \item[(g1)] $g(t,x)=1-t^2-|x|^2$ if $t^2+|x|^2<1/2$,
    
    \smallskip
    
    \item[(g2)] $0\leq g(t,x)\leq 1/2$ if $t^2+|x|^2>1/2$, and $g(t,x)= 0$ if $t^2+|x|^2>1$.
\end{enumerate}
For $\eps>0$, denote $g_\eps(t,x):=g(t/\eps,x/\eps)$, and
\[
 L:=\sup\left\{v(t,x),-v^h(t,x)\,:\,(t,x)\in [0,T]\times\bbR^d\right\}+1\geq 1,
\]
By Lemma \ref{L.2.3}, $\sigma\leq L\leq CT$ for some universal constant $C>0$. Next, for $\phi(x):=(1+|x|^2)^{1/2}$ and $R\geq |x_0|+T$, we define $\Phi^h: [0,T]^2\times\bbR^{2d}\to \bbR$ by
\begin{align*}
\Phi^h(t,s,x,y)&:=v(t,x)-v^h(s,y)-\frac{\sigma}{T}(2T-t-s)\\
&\qquad\quad-\frac{\sigma}{R}(\phi(x)+\phi(y))+(8L+2\sigma)g_\eps(t-s,x-y).
\end{align*}

\quad Since $v,v^h$ are bounded and continuous, there exists $(t_1,s_1,x_1,y_1)\in [0,T]^2\times \bbR^{2d}$ such that
\beq\lb{3.3}
\Phi^h(t_1,s_1,x_1,y_1)=\max_{[0,T]^2\times \bbR^{2d}}\Phi^h(t,s,x,y).
\eeq
Due to $\phi(x_0)\leq R$, by \eqref{3.2},
\beq\lb{3.10}
\Phi^h(t_1,s_1,x_1,y_1)\geq \Phi^h(t_0,t_0,x_0,x_0)\geq 8L+6\sigma.
\eeq
Since $\max\{v(t_1,x_1),-v^h(s_1,y_1)\}\leq L$, we deduce
$\Phi^h(t_1,s_1,x_1,y_1)\leq 2L+(8L+2\sigma)g_\eps(t_1-s_1,x_1-y_1)$,
which, together with \eqref{3.10}, implies
$g_\eps(t_1-s_1,x_1-y_1)\geq {3}/{4}$.
Then by (g1), we get that for some $C>0$, 
\beq\lb{3.4}
g_\eps(t-s,x-y)=1-\eps^{-2}|t-s|^2-\eps^{-2}|x-y|^2,
\eeq
whenever $|t-t_1|,|s-s_1|,|x-x_1|,|y-y_1|\leq \eps/C$.

\quad Now, by \eqref{3.3}, the mapping
\beq\lb{3.5}
(t,x)\mapsto v(t,x)+\frac{\sigma}{T}t-\frac{\sigma}{R}\phi(x)+(8L+2\sigma)g_\eps(t-s_1,x-y_1).
\eeq
is maximized at $(t,x)=(t_1,x_1)$. 
Together with the fact that $v$ is Lipschitz continuous (taking $M:=1+\|\nabla v\|_\infty$) and $|\nabla\phi|\leq 1$, 
we find that 
$|\nabla_x \,g_\eps(t_1-s_1,x_1-y_1)|\leq (M+{\sigma}{R^{-1}})(8L+2{\sigma})^{-1}$
and, 
$|\partial_t \,g_\eps(t_1-s_1,x_1-y_1)|\leq (M+{\sigma}T^{-1})(8L+2\sigma)^{-1}$.
By \eqref{3.4}, $\sigma\leq L\leq CT$ and $R\geq T$, these yield
\beq\lb{3.8}
|x_1-y_1|\leq C\eps^2(M+{\sigma}{R}^{-1})(L+\sigma)^{-1}\leq C\eps^2ML^{-1},
\eeq
and
\beq\lb{3.9}
|t_1-s_1|\leq C\eps^2(M+{\sigma}T^{-1})(L+\sigma)^{-1}\leq C\eps^2M L^{-1}.
\eeq

\quad Now, we firstly assume that $t_1,s_1<T$. In view of \eqref{3.5}, we apply the viscosity solution test for $v$ to get
\beq\lb{4.8}
\begin{aligned}
&-\frac{\sigma}{T}-(8L+2{\sigma})\,\partial_t g_\eps(t_1-s_1,x_1-y_1)\\
&\qquad\qquad+H\left(t_1,x_1,\frac\sigma{R}\nabla\phi(x_1)-(8L+2\sigma)\nabla_x \,g_\eps(t_1-s_1,x_1-y_1))\right)\geq 0.
\end{aligned}
\eeq
Similarly, since
$(s,y)\to v^{h}(s,y)-\frac{\sigma}{T}s+\frac{\sigma}{R}\phi(y)-(8L+2\sigma)g_\eps(t_1-s,x_1-y)$
is minimized at $(s_1,y_1)$, the comparison principle yields
\begin{align*}
&\frac{\sigma}{T}-(8L+2\sigma)\,\partial_t g_\eps(t_1-s_1,x_1-y_1)\\
&\qquad\qquad+H\left(s_1,y_1,-\frac\sigma{R}\nabla^h\phi(y_1)-(8L+2\sigma)\nabla_x^h\, g_\eps(t_1-s_1,x_1-y_1)\right)\\
&\qquad\qquad - Nh\Delta^h \left[\frac{\sigma}{R}\phi(y_1)-(8L+2\sigma)g_\eps(t_1-s_1,x_1-y_1))\right]\leq 0.
\end{align*}
Thus we get
\beq\lb{3.11}
\begin{aligned}
\frac{2\sigma}{T}&\leq H\left(t_1,x_1,\frac\sigma{R}\nabla\phi(x_1)-(8L+2\sigma)\nabla_x\, g_\eps(t_1-s_1,x_1-y_1))\right)\\
&\qquad\qquad-H\left(s_1,y_1,-\frac\sigma{R}\nabla\phi(y_1)-(8L+2\sigma)\nabla_x^h\, g_\eps(t_1-s_1,x_1-y_1)\right)\\
&\qquad\qquad + Nh\Delta^h \left[\frac{\sigma}{R}\phi(y_1)-(8L+2\sigma)g_\eps(t_1-s_1,x_1-y_1))\right].
\end{aligned}
\eeq

\quad It follows from \eqref{3.4} that for $h \ll \eps$, we have at point $(t_1-s_1,x_1-y_1)$,
\beq\lb{3.6'}
\nabla_x^h \,g_\eps=\nabla_x\, g_\eps=2\eps^{-2}(x_1-y_1), \quad \Delta^h g_\eps=-2d\eps^{-2}.
\eeq
Due to $|\nabla\phi|\leq 1$ and $\Delta^h \phi\leq C$, we get
\beq\lb{3.6}
Nh\Delta^h \left[\frac{\sigma}{R}\phi(y_1)-(8L+2\sigma)g_\eps(t_1-s_1,x_1-y_1))\right]\leq CL\eps^{-2}h.
\eeq
Using \eqref{3.11}--\eqref{3.6} and the regularity of $H$ (see \eqref{2.5}), we obtain for some universal $C$,
\begin{align*}
{2\sigma}{T^{-1}}&\leq C\sigma R^{-1}(|\nabla\phi(x_1)|+|\nabla\phi(y_1)|)+CL\eps^{-2}h\\
&\qquad +C(|t_1-s_1|+|x_1-y_1|)\left[1+(8L+2\sigma)|\nabla_x \,g_\eps(t_1-s_1,x_1-y_1)|\right]\\
&\leq C\sigma R^{-1}+CL\eps^{-2}h +C(|t_1-s_1|+|x_1-y_1|)\left(1+L\eps^{-2}|x_1-y_1|\right)
\end{align*}
which, by \eqref{3.8} and \eqref{3.9}, yields
${\sigma}{T^{-1}}\leq C\sigma R^{-1}+CL\eps^{-2}h+C\eps^2M^2L^{-1}$.
Now we take $\eps:=M^{-1/2}L^{1/2}h^{1/4}$ and pass $R\to \infty$. Then when $h$ is sufficiently small, we obtain 
$\sigma\leq  C TM \sqrt{h}$ for some universal $C > 0$.
This finishes the proof of the upper bound of $\sup_{ [0,T]\times\bbR^d }(v-v^h)$ in the case when $t_1,s_1<T$.

\quad Next, suppose that one of $t_1$ and $s_1$ equals to $T$. 
We only prove for the case when $t_1=T$. By \eqref{3.10} and the definition of $\Phi^h$,
\[
8L+6\sigma \leq v(t_1,x_1)-v^h(s_1,y_1)+(8L+2\sigma)g_\eps(t_1-s_1,x_1-y_1).
\]
It follows from the proof Lemma \ref{L.2.3} that $v^h$ is Lipschitz continuous with unit Lipschitz constant when $|T-t|\leq C$. Note that $\eps^2ML^{-1}\leq C$. Hence \eqref{3.8} and \eqref{3.9} yield
\begin{align*}
8L+6\sigma 
&\leq |v(T,x_1)-q(y_1)|+|q(y_1)-v^h(s_1,y_1)|+(8L+2\sigma)g_\eps(T-s_1,x_1-y_1)\\
&\leq C(|x_1-y_1|+|T-s_1|)+8L+2\sigma \leq C\eps^2ML^{-1}+8L+2\sigma.
\end{align*}
This yields $\sigma\leq C \sqrt{h}$ for some universal $C>0$.

\quad Finally, the upper bound estimate for $ \sup_{ [0,T]\times B_R} ( v^h-v)$ follows by using the same argument as the above.
Applying Lemma \ref{L.2.3} permits to conclude.
\end{proof}

\subsection{Almost everywhere convergence of the policy}
It was proved in \cite{vis} that, under suitable assumptions, the solution $v$ to \eqref{3.1} is semi-concave in space. From this, we are able to derive the almost everywhere convergence of the policies. 

\quad Below we say that a function $g:\bbR^d \to\bbR$ is uniformly semi-concave if there exists $C>0$ such that for all $x,y\in \bbR^d $ we have
$g(x+y)+g(x-y)-2g(x)\leq C|y|^2$.
If $g$ is uniformly bounded and Lipschitz continuous, and both $\pm g$ are uniformly semi-concave, then $g$ is bounded in $ W^{2,\infty}(\bbR^d)$ and we denote by 
\[
\|g\|_{W^{2,\infty}}=\|g\|_\infty + \|\nabla g\|_\infty + \|\nabla^2 g\|_\infty.
\]
We make the following assumption:
\begin{itemize}
    \item[(A3)] $q(\cdot)$ is uniformly semi-concave, and $c(t,\cdot,a),f(t,\cdot,a)$ are bounded in $W^{2,\infty}(\bbR^d)$ uniformly in $t\in[0,T]$ and $a\in A$.
\end{itemize}

\begin{theorem}\lb{T.2.8}
Under the assumptions of Theorem \ref{T.2.5}, further assume  {\rm (A3)}, $q\in C^4$, $H\in C^2$, and   when $|p|\leq \sup_{\tau,h}\|\nabla^hV^{\tau,h}\|_{L^{\infty}(\Omega_T^{\tau,h})}$, $|D_pH|\leq N$ and $-D_{pp}H\geq \theta I$ for some $\theta>0$. Then $v^h(t,\cdot)$ and $v(t,\cdot)$ are uniformly semi-concave  for all $t\in [0,T]$. Moreover, for each $t\in[0,T]$ we have for a.e. $x\in\bbR^d$,
\[
\alpha(t_h,x_h,\nabla^h v^h(t_h,x_h ))\to \alpha(t,x,\nabla v(t,x))\quad\text{ as }h\to 0
\]
where $[0,T]\times\bbR^d\ni(t_h,x_h)\to (t,x)$ as $h\to 0$. 
\end{theorem}

\quad The extra assumptions are made to show semi-concavity of $v^h$. In general, we can show a weak type of semi-concavity of $v^h$.

\begin{theorem}\lb{T.3.5}
Under the assumptions of Theorem \ref{T.2.5}, further assume {\rm (A3)}. Then there exists $C>0$ (also depending on {\rm(A3)}) such that for all $h\in (0,1)$, $t\in [0,T]$ and $x,y\in \bbR^d$,
\[
v^h(t,x+y)+v^{h}(t,x-y)-2v^{h}(t,x)\leq C\exp(CT)\,(|y|^2 +\sqrt{h}).
\]
\end{theorem}

\quad The proofs of the two theorems are similar to those of Theorem \ref{T.4.10} and Theorem \ref{T.4.11}, and we choose to write the full details down there (as it is slightly more complicated there).

\section{Analysis of discrete space-time schemes}
\label{sc4}

\subsection{Convergence of PI} 
The parallel result of Theorem \ref{T.2.4} on the convergence of $V_n^{\tau,h}\to V^{\tau,h}$ holds the same
(see Figure \ref{fig} for a numerical illustration).
However, the proof is more involved due to the discretization in the time direction. In it, we will emphasize the difference.

\begin{theorem}\lb{T.4.2}
Assume  {\rm (A1)--(A2)} and $N\geq 1$. Let $V_n:=V^{\tau,h}_n$ and $V:=V^{\tau,h}$ be, respectively, continuous solutions to \eqref{4.1} and \eqref{4.2}. Then there exists a universal constant $C>0$ such that if $C(1+\|\nabla^h V\|^2_\infty)\tau\leq h$,
we have for all $n\geq 1$, $R\geq 1$ and $t\in \bbN_{T}^\tau$,
\[
\begin{aligned}
&\sum_{x\in\bbZ_h^d, |x|\leq R}\left|V_n(t,x)-V(t,x)\right|^2 \leq
\frac{h\tau}{2^{n+1}} \sum_{t\leq s\in \bbN_{T}^\tau}\sum_{ x\in \bbZ^d_h}  \\
&\qquad\qquad  \exp\left[{C\exp(1+\|\nabla^h V\|^2_\infty)(s-t)/h}\right] \left|D^h(V_0(s,x)-V(s,x))\right|^2\min\left\{1,e^{-|x|+R+1}\right\}.
\end{aligned}
\]
In particular, we have
\[
\max_{t\in\bbN_T^\tau}\sum_{x\in\bbZ_h^d, |x|\leq R}\left|V_n(t,x)-V(t,x)\right|^2 \leq C2^{-n}\exp\left[C\exp(CT)/h\right]R^d,
\]
\[
\max_{t\in\bbN_T^\tau}\sum_{ x\in \bbZ^d_h,|x|\leq R}\left|\alpha(t,x,\nabla^h V_n(t,x))- \alpha(t,x,\nabla^h V(t,x))\right|^2\leq C2^{-n}\exp\left[C\exp(CT)/h\right]R^d.
\]
\end{theorem}

\begin{proof}
Assume $T\geq 1$ for simplicity. Let $\varphi=\varphi_R:[0,\infty)\to [0,1]$ be $C^1$ and satisfying \eqref{p.8}, and let $A:=CT^2/h$ for some $C>0$ to be determined. Then for $t\in \bbN_{T}^\tau$ set 
\[
E_{t,n}:=\frac12 e^{At}\sum_{x\in \bbZ^d_h}|V_n(t,x)-V(t,x)|^2\varphi(|x|)
\]
which is finite. Direct computation yields
\beq\lb{4.4}
\begin{aligned}
\frac{E_{t,n}-E_{t-\tau ,n}}\tau&\geq Ae^{-A\tau}E_{t,n}+\frac{1}{2}e^{A(t-\tau )}\sum_{x\in  \bbZ^d_h}(V_n(t,x)+V_n(t-\tau ,x)-V(t,x)-V(t-\tau ,x))\times\\
&\qquad \qquad (\partial_t^\tau V_n(t,x)-\partial_t^\tau V(t,x))\varphi(|x|)\\
&= Ae^{-A\tau}E_{t,n}+e^{A(t-\tau )}\sum_{x\in  \bbZ^d_h}(V_n(t,x)-V(t,x))(\partial_t^\tau V_n(t,x)-\partial_t^\tau V(t,x))\varphi(|x|)\\
&\qquad \qquad -\frac{\tau}2 e^{A(t-\tau )}\sum_{x\in \bbZ^d_h}\left|\partial_t^\tau V_n(t,x)-\partial_t^\tau V(t,x)\right|^2\varphi(|x|)\\
&=:Ae^{-A\tau}E_{t,n}+e^{A(t-\tau )}X_{t,n}-\frac{\tau}2 e^{A(t-\tau )}Y_{t,n}.
\end{aligned}
\eeq

\quad First, we consider the term $Y_{t,n}$ (which does not appear in the semi-discretization problem in Theorem \ref{T.2.4}). Similarly as before, for simplicity of notations, we write 
\[
\alpha:=\alpha(t,x,\nabla^h V(t,x)),\quad \alpha_n:=\alpha(t,x,\nabla^h V_{n-1}(t,x)),
\]
\[
 c_n:=c(t,x,\alpha_n(t,x)) \quad \text{and}\quad f_n:=f(t,x,\alpha_n(t,x))).
\]
We will also drop $(t,x)$ from the notations of $V(t,x),V_n(t,x)$, and $(|x|)$ from $\varphi(|x|)$.
It follows from the equations \eqref{4.1} and \eqref{4.2} that
\begin{align*}
    Y_{t,n}&=\sum_{x\in \bbZ^d_h}\left|c_n+\nabla^h V_n\cdot f_n+Nh\Delta^h V_n-H(t,x,\nabla^h V)-Nh\Delta^h V\right|^2\varphi(|x|)
\end{align*}
Recall that 
$H(t,x,\nabla^h V)=c(t,x,\alpha)+f(t,x,\alpha)\cdot \nabla^h V$,
and $|\nabla^h V|\leq M$ for some $M\geq 1$. 
So, the regularity assumptions and \eqref{p.6} yield
\begin{align*}
Y_{t,n}&\leq C \sum_{x\in \bbZ^d_h}\left(M^2|D^h V_{n-1}-D^h V|^2+M^2|D^{-h} V_{n-1}-D^{-h} V|^2
+\right.\\
&\qquad\qquad \left.|D^h V_{n}-D^h V|^2+|D^{-h} V_{n}-D^{-h} V|^2\right)\varphi(|x|) \le C \left(M^2G_{t,n-1}^h+ G_{t,n}^h\right),
\end{align*}
where in the last inequality we used the notation
$G_{t,n}^h:=\sum_{x\in \bbZ^d_h}|D^hV_n(t,x)-D^hV(t,x)|^2\varphi(|x|)$.
and
\eqref{p.13} with the above defined $G_{t,n}^h$ (which clearly holds the same).

\quad Next, we consider the term $X_{t,n}$. 
Note that for any $v\in L^\infty(\bbZ_h^d)$,
\[
\sum_{x\in \bbZ^d_h}\Delta^h v\, v\varphi =-\sum_{x\in \bbZ^d_h}|D^h v|^2\varphi-\sum_{x\in \bbZ^d_h} vD^{-h} v\cdot D^{-h}\varphi.
\]
So, similarly as before (also using
the equation, \eqref{p.8}, \eqref{p.13}, the uniform Lipschitz assumptions, and Young's inequality), we have for some universal $C>0$ and for any $\sigma_1,\sigma_2>0$,
\begin{align*}
    X_{t,n} &\geq  Nh G_{t,n}^h-C\sum_{x\in \bbZ^d_h} |V_n-V|(|D^h(V_{n}-V)|+|D^{-h}(V_{n}-V)|)\varphi\\
&\qquad\qquad-CM\sum_{x\in \bbZ^d_h} |V_n-V|(|D^h(V_{n-1}-V)|+|D^{-h}(V_{n-1}-V)|)\varphi\\
&\geq  (Nh-\sigma_1)G_{t,n}^h-\sigma_2G_{t,n-1}^h-C(\sigma_1^{-1}+M^2\sigma_2^{-1})e^{-At}E_{t,n}.
\end{align*}

\quad Since $E_{T,n}\equiv 0$, putting the above together and summing up \eqref{4.4} with respect to $t$ yield
\beq\lb{p.5'}
\begin{aligned}
-E_{t,n}/\tau &\geq ( Nh-\sigma_1-C\tau)\sum_{t+\tau\leq s\in \bbN_{T}^\tau} e^{A(s-\tau)}G_{s,n}^h-(\sigma_2+CM^2\tau)\sum_{t+\tau\leq s\in \bbN_{T}^\tau} e^{A(s-\tau)}G_{s,n-1}^h\\
&\qquad +(A-C\sigma_1^{-1}-CM^2\sigma_2^{-1})e^{-A\tau}\sum_{t+\tau\leq s\in \bbN_{T}^\tau} E_{s,n}^h
\end{aligned}
\eeq
for some universal constant $C>0$.

\quad Finally we take 
$\sigma_1:=h/4$, $\sigma_2:=h/8$, $A:=12CM^2/h$.
Then if $\tau\leq h/(8CM^2)$,
\eqref{p.5'} yields
\[
\sum_{t+\tau\leq s\in \bbN_{T}^\tau}e^{A(s-\tau)} G_{s,n}^h\leq \frac12\sum_{t+\tau\leq s\in \bbN_{T}^\tau} e^{A(s-\tau)}G_{s,n-1}^h\leq\ldots\leq 2^{-n}\sum_{t+\tau\leq s\in \bbN_{T}^\tau}e^{A(s-\tau)} G_{s,0}^h,
\]
and then
$
E_{t,n}\leq \frac{  h\tau}{4}\sum_{t+\tau\leq s\in \bbN_{T}^\tau}e^{A(s-\tau)}G_{s,n-1}^h\leq  \frac{h\tau}{2^{n+1}}\sum_{t+\tau\leq s\in \bbN_{T}^\tau}e^{A(s-\tau)}G_{s,0}^h$.
This, together with Proposition \ref{P.4.1}, concludes the proof of the first claim as before. 

\quad The second claim follows similarly as in Theorem \ref{T.2.add}.
\end{proof}

\quad By shifting the solutions, we obtain uniform pointwise exponential convergence of $V^{\tau,h}_n$ to $V^{\tau,h}$ and $\alpha(\cdot,\cdot,\nabla^h V^{\tau,h}_n(\cdot,\cdot))$ to $\alpha(\cdot,\cdot,\nabla^h V^{\tau,h}(\cdot,\cdot))$ as $n\to \infty$, in $\Omega_T^{\tau,h}$.


\subsection{Convergence of $V^{\tau,h}$ as $\tau, h \to 0$.}
Let $V^{\tau,h}$ and $v$ be, respectively, solutions to \eqref{4.2} and \eqref{3.1}. The following theorem proves that the difference between $V^{\tau,h}$ and $v$ is at most of order $\sqrt{h}$. The argument follows the idea of \cite[Theorem 1]{crandall1984two}, which considered the discrete space-time scheme for the homogeneous Hamilton-Jacobi equation $v_t+H(Dv)=0$. 

\begin{theorem}\lb{T.4.3}
Assume  {\rm (A1)--(A2)} and \eqref{N}. Then there exists a universal $C>0$ such that
\[
\sup_{(t,x)\in \Omega^{\tau,h}_T} |v(t,x)-V^{\tau,h}(t,x)|\leq C(1+T)(1+\|\nabla v\|_\infty) \sqrt{h}.
\]
In particular, we have
$\sup_{(t,x)\in \Omega^{\tau,h}_T} |v(t,x)-V^{\tau,h}(t,x)|\leq C\exp(CT) \sqrt{h}$.
\end{theorem}

\begin{remark}\lb{R.4}
It was shown in \cite{10DonKry,DonKry,18Kry} that 
\[
\sup_{(t,x)\in \Omega^{\tau,h}_T} |v(t,x)-V^{\tau,h}(t,x)|\leq C(\tau^{1/4}+h^{1/2})\quad\text{ for some $C=C(T)>0$},
\]
where $v$ solves a general degenerate parabolic Bellman equation and $V^{\tau,h}$ is its space-time finite difference approximation.  
For the first order equations, our Theorem \ref{T.4.3} obtains a better convergence rate of $C(\tau^{1/2}+h^{1/2})$.
\end{remark}

\begin{proof}
Assume $T\geq 1$. And suppose for some $(t_0,x_0)\in \Omega^{\tau,h}_T$ such that
\beq\lb{4.5}
8\sigma:=v(t_0,x_0)-V^{\tau,h}(t_0,x_0)\geq \frac12 \sup_{(t,x)\in \Omega^{\tau,h}_T}\left[ v(t,x)-V^{\tau,h}(t,x)\right]>0.
\eeq
Let $D_{T,\tau,h}:=[0,T]\times \bbN_{T}^\tau\times \bbR^d\times \bbZ^d_h$, and
\[
L:=\sup\left\{v(t,x),-V^{\tau,h}(t,x)\,:\,(t,x)\in \Omega^{\tau,h}_T\right\}+1.
\]
Then $\sigma\leq L\leq CT$ for some universal constant $C>0$. Moreover, let $R$, $g$ and $g_\eps$ with $\eps\in (0,1)$, and $\phi$ be from the proof of Theorem \ref{T.2.5}, and define $\Phi^h: D_{T,\tau,h}\to \bbR$ by
\begin{align*}
\Phi^h(t,s,x,y)&:=v(t,x)-V^{\tau,h}(s,y)-\frac{\sigma}{T}(2T-t-s)\\
&\qquad\quad-\frac{\sigma}{R}(\phi(x)+\phi(y))+(8L+2\sigma)g_\eps(t-s,x-y).
\end{align*}
Suppose
\beq\lb{4.6}
\Phi^h(t_1,s_1,x_1,y_1)=\max_{D_{T,\tau,h}}\Phi^h(t,s,x,y).
\eeq
It is clear that \eqref{3.10}--\eqref{3.9} hold the same. By \eqref{3.9} if $\tau \ll \eps^2M/L$, we get
\beq\lb{4.9}
|t_1-s_1-\tau|\leq C\eps^2M/L\quad\text{ with }M=1+\|\nabla v\|_\infty.
\eeq

\quad First, assume $t_1,s_1<T$. The viscosity solution test for $v$ shows \eqref{4.8} by \eqref{3.5}. 
Next since
$\Omega^{\tau,h}_T\ni(s,y)\to V^{\tau,h}(s,y)-\frac{\sigma}{T}s+\frac{\sigma}{R}\phi(y)-(8L+2\sigma)g_\eps(t_1-s,x_1-y)$
is minimized at $(s_1,y_1)$, then for all $(s,y)\in \Omega^{\tau,h}_T$,
\begin{align*}
    V^{\tau,h}(s,y)&\geq V^{\tau,h}(s_1,y_1)-\frac{\sigma}{T}(s_1-s)+\frac{\sigma}{R}(\phi(y_1)-\phi(y))\\
    &\qquad\qquad-(8L+2\sigma)\left[g_\eps(t_1-s_1,x_1-y_1)-g_\eps(t_1-s,x_1-y)\right]=:\tilde V(s,y).
\end{align*}
Recall that $s_1+\tau\leq T$ and $\calF_t$ from \eqref{calF} satisfies the monotonicity property. We obtain
\[
V^{\tau,h}(s_1,y_1)=\calF_{s_1+\tau}(V^{\tau,h}(s_1+\tau,\cdot))(y_1)\geq\calF_{s_1+\tau}(\tilde{V}(s_1+\tau,\cdot))(y_1),
\]
which gives
\beq\lb{4.10}
\begin{aligned}
&0\geq \frac{\sigma}{T}-(8L+2\sigma)\,\partial^\tau_t g_\eps(t_1-s_1,x_1-y_1)\\
&\qquad\qquad+H\left(s_1+\tau,y_1,-\frac\sigma{R}\nabla^h\phi(y_1)-(8L+2\sigma)\nabla_x^h\, g_\eps(t_1-s_1-\tau,x_1-y_1)\right)\\
&\qquad\qquad - Nh\Delta^h \left[\frac{\sigma}{R}\phi(y_1)-(8L+2\sigma)g_\eps(t_1-s_1-\tau,x_1-y_1))\right].
\end{aligned}
\eeq

\quad By \eqref{3.4}, if $\tau,h \ll \eps^2$,
\beq\lb{4.11}
|\partial_t^\tau \,g_\eps(t_1-s_1,x_1-y_1)-\partial_t g_\eps(t_1-s_1,x_1-y_1)|\leq C\eps^{-2}\tau,
\eeq
\beq\lb{4.12}
\nabla_x^h \,g_\eps(t_1-s_1-\tau,x_1-y_1)=\nabla_x\, g_\eps(t_1-s_1,x_1-y_1)=2\eps^{-2}(x_1-y_1).
\eeq
Combining \eqref{4.10} with \eqref{4.8}, and using \eqref{4.11} and \eqref{4.12} yield
\beq\lb{4.32}
\begin{aligned}
\frac{2\sigma}{T}&\leq H\left(t_1,x_1,\frac\sigma{R}\nabla\phi(x_1)-(8L+2\sigma)2\eps^{-2}(x_1-y_1)\right)\\
&\qquad\quad-H\left(s_1+\tau,y_1,-\frac\sigma{R}\nabla\phi(y_1)-(8L+2\sigma)2\eps^{-2}(x_1-y_1)\right)\\
&\qquad\quad + Nh\Delta^h \left[\frac{\sigma}{R}\phi(y_1)-(8L+2\sigma)g_\eps(t_1-s_1-\tau,x_1-y_1))\right]+CL\eps^{-2}\tau.
\end{aligned}
\eeq
The definitions of $\phi$ and $g_\eps$ show \eqref{3.6}.
Then, applying \eqref{2.5} and \eqref{3.6} into \eqref{4.32}, if $(\tau\leq)\, h\ll \eps^{2}$ we deduce for some $C>0$ that
\beq\lb{4.36}
\begin{aligned}
\sigma{T^{-1}}&\leq  C\sigma R^{-1}(|\nabla\phi(x_1)|+|\nabla\phi(y_1)|)+CL\eps^{-2}h+CL\eps^{-2}\tau\\
&\qquad +C(|t_1-s_1-\tau|+|x_1-y_1|)\left[1+(8L+2\sigma)2\eps^{-2}|x_1-y_1|\right]\\
&\leq C\sigma R^{-1}+CL\eps^{-2}h+C\eps^2M^2 L^{-1}
\end{aligned}
\eeq
where in the second inequality we also used \eqref{3.8} and \eqref{4.9}.

\quad Now we take $\eps:=M^{-1/2}L^{1/2}h^{1/4}$, and send $R\to\infty$. 
It is clear that $\tau\ll \eps^2M/L$ is satisfied when $h$ is small.
We obtain from \eqref{4.36} that $\sigma\leq CTM\sqrt{h}$, which finishes the proof of the upper bound of $\sup_{ \Omega^{\tau,h}_{T} }(v-V^{\tau,h})$ in the case when $t_1,s_1<T$. 

\quad Next, if at least one of $t_1$ and $s_1$ equals to $T$, the argument of Theorem \ref{T.2.5} applies the same except that we need to use Proposition \ref{P.4.1} in place of Lemma \ref{L.2.3}.
Finally, the proof for the upper bound of $ \sup_{\Omega^{\tau,h}_{T}} (V^{\tau,h}-v)$ is the same.
\end{proof}

\subsection{Almost everywhere convergence of the policy}
We show the almost everywhere convergence of the policy and some semi-concavity property of the solution.

\begin{theorem}\lb{T.4.10}
Under the assumptions of Theorem \ref{T.4.3}, further assume {\rm (A3)}, $q\in C^4$, $H\in C^2$, and  when $|p|\leq \sup_{\tau,h}\|\nabla^hV^{\tau,h}\|_{L^{\infty}(\Omega_T^{\tau,h})}$, $|D_pH|\leq N$ and $-D_{pp}H\geq \theta I$ for some $\theta>0$. Then $V^{\tau,h}$ and $v$ are uniformly semi-concave for all $t\in [0,T]$. 
Moreover, for each $t\in[0,T]$ we have for a.e. $x\in\bbR^d$,
\[
\alpha(t_{h},x_{h},\nabla^h V^{\tau_h,h}(t_{h},x_{h} ))\to \alpha(t,x,\nabla v(t,x))\quad\text{ as }h\to 0
\]
where $\Omega_T^{\tau_h,h}\ni(t_{h},x_{h})\to (t,x)$ as $h\to 0$ and $\tau_h$ satisfies $0<2N\tau_h\leq h$. 
\end{theorem}

\begin{proof}

The semi-concavity of $v(t,\cdot)$ follows from \cite{vis}. 

\quad Next, we show semi-concavity of $V^{\tau,h}(t,\cdot)$. By replacing $q(x)$ by $q(x+y)$ for $y\in [0,h]^d$, we can extend $V^{\tau,h}$ to be a function on $\bbN_T^\tau\times\bbR^d$. Differentiating \eqref{4.2} twice in $x_i$ to get
\[
\partial_t^\tau V^{\tau,h}_{ii} + D_pH\cdot\nabla^h V^{\tau,h}_{ii} +Nh\Delta^h V^{\tau,h}_{ii}+H_{ii}+2D_pH_{i}\cdot\nabla^h V^{\tau,h}_{i}+\sum_{k,l}H_{p_kp_l}\nabla_k^hV^{\tau,h}_{i}\nabla_l^hV^{\tau,h}_{i}=  0.
\]
Here we used the notation that for a function $f:\bbR^d\to\bbR$, $\nabla_k^h f:=\frac{f(x+he_k)-f(x-he_k)}{2h}$, $f_i:=\partial_{x_i}f$ and  $f_{ii}:=\partial_{x_i}^2f$. By the assumptions, there exists $C>0$ such that
\[
\sum_{k,l}H_{p_kp_l}\nabla_k^hV^{\tau,h}_{i}\nabla_l^hV^{\tau,h}_{i}\leq -\theta|\nabla^h V^{\tau,h}_{i}|^2,
\]
\[
|H_{ii}|\leq C\quad\text{and}\quad 2|D_pH_{i}\cdot\nabla^h V^{\tau,h}_{i}|\leq \theta|\nabla^h V^{\tau,h}_{i}|^2+C.
\]
Hence, we get
\[
\partial_t^\tau V^{\tau,h}_{ii} + D_pH\cdot\nabla^h V^{\tau,h}_{ii} +Nh\Delta^h V^{\tau,h}_{ii}\geq -C.
\]
Since $D_pH$ is uniformly finite, $q\in C^4$ and $N\geq |D_pH|$,  the comparison principle yields
\[
V^{\tau,h}_{ii}\leq \|q_{ii}\|_\infty+C(T-t).
\]
This implies that $V^{\tau,h}$ is uniformly semi-concave in $x$ for $t\in\bbN_T^\tau$.

\quad For the second claim, it suffices to prove that for a fixed $t\in[0,T)$, and a.e.  $x\in \bbR^d$,
\beq\lb{4.31}
\nabla^h V^{\tau_h,h}(t_{h},x_{h} )\to \nabla v(t,x)\quad\text{ as }h\to 0.
\eeq
For any function $g:\bbR^d\to\bbR$, we denote by $D^{+}g(x)$ the set of subdifferential of $g$:
\[
D^+g(x):=\left\{p\in\bbR^d\,\big|\, \limsup_{y\to x}\frac{g(y)-g(x)-p\cdot (y-x)}{|y-x|}\leq 0\right\}.
\]
Due to $v(t,\cdot)$ is semi-concave, $D^+v(t,x)$ is non-empty for all $x\in\bbR^d$.

\quad Because $v(t,\cdot)$ is Lipschitz continuous, $\nabla_x v(t,x)$ exists for a.e. $x\in \bbR^d$. We fix one such $x$.
Since $ V^{\tau_h,h}$ are Lipschitz continuous uniformly in $h$, after passing to a subsequence of $h\to 0$, we can assume that 
$\nabla^h V^{\tau_h,h}(t_{h},x_{h} )\to p$ 
for some $p\in\bbR^d$. 
Since $V^{\tau_h,h}(t_{h},x_{h} )\to  v(t,x)$ as $h\to 0$, and both $V^{\tau_h,h}$ and $v$ are semi-concave in $x$, the stability of subdifferential yields that $p\in D^+v(t,x)$. While because $\nabla_x v(t,x)$ exists, we get $p=\nabla_x v(t,x)$. Note that this is for any convergent subsequence of $\nabla^h V^{\tau_h,h}(t_{h},x_{h} )$, and so we obtain \eqref{4.31}.
\end{proof}

\quad Below, we show a weak type of semi-concavity of $V^{\tau,h}(t,\cdot)$. 
We use the ``doubling variable'' method, see e.g., \cite{vis}.

\begin{theorem}\lb{T.4.11}
Under the assumptions of Theorem \ref{T.4.3}, further assume {\rm (A3)}. Then there exists $C>0$ (also depending on {\rm(A3)}) such that for all $t\in \bbN_{T}^\tau$ and $x,y\in \bbZ_h^d$,
\[
V^{\tau,h}(t,x+y)+V^{\tau,h}(t,x-y)-2V^{\tau,h}(t,x)\leq  C\exp(CT)\,(|y|^2 +\sqrt{h}).
\]
\end{theorem}

\begin{proof}
It suffices to show that there exist $C_T,C_T'>0$ depending on the assumptions such that
\beq\lb{4.17}
V^{\tau,h}(t,x)+V^{\tau,h}(t,z)-2V^{\tau,h}(t,y)\leq C_T\left(|x-y|^2+|z-y|^2+|x+z-2y|\right) +C_T'\sqrt{h}
\eeq
for all $t\in \bbN_{T}^\tau$ and $x,y,z\in \bbZ_h^d$. By the assumption on $q$, the inequality holds for $t=T$ with $C_T=\|q\|_{W^{2,\infty}}=:C_0$, and $C_T'=0$.

\quad Suppose for contradiction that \eqref{4.17} fails. Then we have for some $C_1\geq 1$ to be determined, and some $C\geq 2$,
\beq\lb{4.51}
\begin{aligned}
&V^{\tau,h}(t,x)+V^{\tau,h}(t,z)-2V^{\tau,h}(t,y)\\
&\qquad\quad- 2C_0e^{C_1(T-t)}\left(|x-y|^4+|z-y|^4+|x+z-2y|^2\right)^{1/2}\geq Ce^{C_1(T-t)}\sqrt{h}
\end{aligned}
\eeq 
for some $(t,x,y,z)=(t',x',y',z')\in \bbN_{T}^\tau\times \bbZ_h^d$. Since $V^{\tau,h}(t,\cdot)$ is Lipschitz continuous (with Lipschitz constant bounded by $C\exp(C(T-t))$ by Proposition \ref{P.4.1} with a shift in time), after enlarging the constant $C$ in \eqref{4.51} and replacing $y'$ by $y'+\sqrt{h}$ if necessary, we can assume that 
\beq\lb{4.40}
|x'+z'-2y'|\geq \sqrt{h}.
\eeq

\quad We denote
$\psi(x,y,z):=|x-y|^4+|z-y|^4+|x+z-2y|^2$,
and by \eqref{4.40},
\[
\delta:=\psi(x',y',z')^{1/2}\geq\sqrt{h}.
\]
Then for all $\eps>0$ sufficiently small, we obtain from \eqref{4.51} that
\begin{align*}
\Phi(t,x,y,z)&:=e^{C_1t}\left(V^{\tau,h}(t,x)+V^{\tau,h}(t,z)-2V^{\tau,h}(t,y)\right)- C_0e^{C_1T}\left(\delta+\delta^{-1}\psi(x,y,z)\right)-\eps|y|^2
\end{align*}
satisfies $\Phi(t',x',y',z')\geq e^{C_1T}\sqrt{h}$. With the positive $\eps$-term, $\Phi$ obtains its positive maximum that is at least $e^{C_1T}\sqrt{h}$ in $\Omega_T^{\tau,h}$ at some point $(t_0,x_0,y_0,z_0)\in \bbN_{T}^\tau\times\bbZ_h^d$, where $(t_0,x_0,y_0,z_0)$ depends on $\eps$ and $\delta$. It is clear that $t_0\leq T-\tau$ by the choice of $C_0$. 
Moreover, for
\[
\gamma_0:=\delta+\delta^{-1}\psi(x_0,y_0,z_0),
\]
we have
\beq\lb{4.18}
V^{\tau,h}(t_0,x_0)+V^{\tau,h}(t_0,z_0)-2V^{\tau,h}(t_0,y_0)\geq C_0 e^{C_1(T-t_0)}\gamma_0+e^{C_1(T-t_0)}\sqrt{h}.
\eeq
Due to the uniform boundedness of $V^{\tau,h}$, by further taking $\eps$ to be small enough depending on $C,T$ and $h$, it is easy to get $\eps|y_0|\leq h$.

\quad 
Now since
$\Omega^{\tau,h}_T\ni(t,x)\to e^{C_1t}V^{\tau,h}(t,x)- C_0e^{C_1T}\delta^{-1}\left(|x-y_0|^4+|x+z_0-2y_0|^2\right)$
is maximized at $(t_0,x_0)$, we get for all $(t,x)\in \Omega^{\tau,h}_T$ that
\begin{align*}
    V^{\tau,h}(t,x)&\leq e^{C_1(t_0-t)}V^{\tau,h}(t_0,x_0)+C_0e^{C_1(T-t)}\delta^{-1}\left(|x-y_0|^4+|x+z_0-2y_0|^2\right)\\
    &\quad-C_0e^{C_1(T-t_0)}\delta^{-1}\left(|x_0-y_0|^4+|x_0+z_0-2y_0|^2\right)=:\tilde V(t,x).
\end{align*}
Due to the equation and the monotonicity property of $\calF_t$ (defined in \eqref{calF}),
$V^{\tau,h}(t_0,x_0)=\calF_{t_0+\tau}(V^{\tau,h}(t_0+\tau,\cdot))(x_0)\leq\calF_{t_0+\tau}(\tilde{V}(t_0+\tau,\cdot))(x_0)$.
By direct computation, 
\[
\nabla_x^h (|x-y_0|^4+|x+z_0-2y_0|^2)=4(|x-y_0|^2+h^2)(x-y_0)+2(x+z_0-2y_0),
\]
\[
\Delta_x^h (|x-y_0|^4+|x+z_0-2y_0|^2)=(8+4d)|x-y_0|^2+2dh^2+2d.
\]
We then get
\beq\lb{4.21}
\begin{aligned}
&\frac{(1-e^{-C_1\tau})}\tau V^{\tau,h}(t_0,x_0)\leq  H\left(t_0+\tau,x_0,\nabla_x^h\tilde{V}(t_0+\tau,x_0) \right )+Nh\Delta_x^h \tilde{V}(t_0+\tau,x_0)\\
&\qquad \leq    H\left(t_0+\tau,x_0,2C_{T,\delta}(q_{x_0}+p_0) \right)+CC_{T,\delta} h(|x_0-y_0|^2+1)
\end{aligned}
\eeq
where 
\[
q_{x_0}:=2(|x_0-y_0|^2+h^2)(x_0-y_0),
\]
\beq\lb{4.33}
\begin{aligned}
C_{T,\delta}&:=C_0e^{C_1(T-t_0-\tau)}/\delta \quad\text{and}\quad p_0:=x_0+z_0-2y_0.
\end{aligned}
\eeq

\quad Similarly, since
$\Omega^{\tau,h}_T\ni(t,z)\to e^{C_1t}V^{\tau,h}(t,z)- C_0e^{C_1T}\delta^{-1}(|z-y_0|^4+|x_0+z-2y_0|^2)$
is maximized at $(t_0,z_0)$, we get
\beq\lb{4.22}
\begin{aligned}
\frac{(1-e^{-C_1\tau})}\tau V^{\tau,h}(t_0,z_0)\leq H\left(t_0+\tau,z_0,2C_{T,\delta}(q_{z_0}+p_0) \right)+CC_{T,\delta} h(|z_0-y_0|^2+1).
\end{aligned}
\eeq
where 
$q_{z_0}:=2(|z_0-y_0|^2+h^2)(z_0-y_0)$.

\quad Next, note that
$\Omega^{\tau,h}_T\ni(t,y)\to 2e^{C_1t}V^{\tau,h}(t,y)+ C_0 e^{C_1T}\delta^{-1}\psi(x_0,y,z_0)+\eps |y|^2$
is minimized at $(t_0,y_0)$. Hence we get
$
V^{\tau,h}(t_0,y_0)\geq\calF_{t_0+\tau}(\hat{V}(t_0+\tau,\cdot))(y_0)
$
where 
\begin{align*}
\hat{V}(t,y)&:=e^{C_1(t_0-t)}V^{\tau,h}(t_0,y_0)-(\eps/2)|y|^2+(\eps/2)|y_0|^2\\
&-(C_0/2)e^{C_1(T-t)}\delta^{-1}\psi(x_0,y,z_0)
    +(C_0/2)e^{C_1(T-t_0)}\delta^{-1}\psi(x_0,y_0,z_0).
\end{align*}
From this we obtain
\[
\begin{aligned}
&-\frac{(1-e^{-C_1\tau})}\tau  V^{\tau,h}(t_0,y_0)
\leq - H\left(t_0+\tau,y_0,2C_{T,\delta}(q_{y_0}+p_0)-\eps y_0 \right)-Nh\Delta_y^h\hat{V}(t_0+\tau,y_0)\\
&\qquad\leq - H\left(t_0+\tau,y_0,2C_{T,\delta}(q_{y_0}+p_0)-\eps y_0 \right)+CC_{T,\delta} h(|x_0-y_0|^2+|z_0-y_0|^2+1)+Ch\eps
\end{aligned}
\]
where 
\[
q_{y_0}:=(|x_0-y_0|^2+h^2)(x_0-y_0)+(|z_0-y_0|^2+h^2)(z_0-y_0),
\]
and $C_{T,\delta} $ and $p_0$ are given in \eqref{4.33}. Using $|H_p|\leq C$ and $\eps|y_0|\leq h$ yields
\beq\lb{4.23}
\begin{aligned}
-\frac{(1-e^{-C_1\tau})}\tau & V^{\tau,h}(t_0,y_0)
\leq - H\left(t_0+\tau,y_0,2C_{T,\delta}(q_{y_0}+p_0) \right)\\
&+CC_{T,\delta} h(|x_0-y_0|^2+|z_0-y_0|^2+1)+Ch
\end{aligned}
\eeq

\quad 
Now let $\alpha\in\calA$ be such that
\begin{align*}
H\left(t_0+\tau,y_0,2C_{T,\delta} (q_{y_0}+p_0)\right)=c(t_0+\tau,y_0,\alpha)+2C_{T,\delta}  f(t_0+\tau,y_0,\alpha)\cdot (q_{y_0}+p_0).
\end{align*}

By \eqref{p.15}, denoting $c_\alpha(\cdot):=c(t_0+\tau,\cdot,\alpha)$ and $f_\alpha(\cdot):=f(t_0+\tau,\cdot,\alpha)$, we have
\beq\lb{4.24}
\begin{aligned}
&H\left(t_0+\tau,x_0,2C_{T,\delta} (q_{x_0}+p_0) \right)+H\left(t_0+\tau,z_0,2C_{T,\delta} (q_{z_0}+p_0) \right)\\
&\qquad\quad-2H\left(t_0+\tau,y_0, 2C_{T,\delta} (q_{y_0}+p_0)\right)\\
&
\leq   c_\alpha(x_0)+c_\alpha(z_0)-2c_\alpha(y_0)+2C_{T,\delta} \left[f_\alpha(x_0)\cdot (q_{x_0}+p_0)\right.\\
&\qquad\quad+ \left. f_\alpha(z_0)\cdot (q_{x_0}+p_0)-2f_\alpha(y_0)\cdot (q_{y_0}+p_0)\right]\\
&=   c_\alpha(x_0)+c_\alpha(z_0)-2c_\alpha(y_0)+2C_{T,\delta} \left[ (f_\alpha(x_0)-f_\alpha(y_0))\cdot q_{x_0}+(f_\alpha(z_0)-f_\alpha(y_0))\cdot q_{z_0}+\right.\\
&\qquad\quad+ \left. (f_\alpha(x_0)+f_\alpha(z_0)-2f_\alpha(y_0))\cdot p_0\right]\\
&\leq \|c_\alpha\|_{W^{2,\infty}}(|x_0-y_0|^2+|z_0-y_0|^2+|x_0+z_0-2y_0|)\\
&\qquad\quad+2C_{T,\delta} \|f_\alpha\|_{\Lip}(|x_0-y_0||q_{x_0}|+|z_0-y_0||q_{z_0}|)\\
&\qquad\quad+ 2C_{T,\delta} \|f_\alpha\|_{W^{2,\infty}}(|x_0-y_0|^2+|z_0-y_0|^2+|x_0+z_0-2y_0|)|x_0+z_0-2y_0|,
\end{aligned}
\eeq
where we used $2q_{y_0}=q_{x_0}+q_{z_0}$ and that for any $x,y,z\in\bbR^d$ and $g\in W^{2,\infty}(\bbR^d)$, 
$|g(x)+g(z)-2g(y)|\leq \|g\|_{W^{2,\infty}}(|x-y|^2+|z-y|^2+|x+z-2y|)$.
By Young's inequality, we get
\[
|x_0-y_0||q_{x_0}|+|z_0-y_0||q_{z_0}|\leq 2|x_0-y_0|^4+2|z_0-y_0|^4+h^4.
\]
Also using the definitions of $C_{T,\delta}$ and $\psi$, we get the left-hand side of \eqref{4.24}
\[\leq Ce^{C_1(T-t_0)}(\delta+\delta^{-1}\psi(x_0,y_0,z_0)+h^4)=Ce^{C_1(T-t_0)}\gamma_0+Ce^{C_1(T-t_0)}h^4/\delta
\]
with $C>0$ only depending on $\|q\|_{W^{2,\infty}}$, $\|c_\alpha\|_{W^{2,\infty}}$ and $\|f_\alpha\|_{W^{2,\infty}}$.

\quad Now summing up \eqref{4.21}, \eqref{4.22} and twice of \eqref{4.23}, we get
\begin{align*}
& \frac{(1-e^{-C_1\tau})}\tau\left[V^{\tau,h}(t_0,x_0)+V^{\tau,h}(t_0,z_0)-2V^{\tau,h}(t_0,y_0)\right]\\
&\qquad\qquad\leq Ce^{C_1(T-t_0)} \gamma_0+Ce^{C_1(T-t_0)}h^4/\delta+CC_{T,\delta} h(|x_0-y_0|^2+|z_0-y_0|^2+1)+Ch\\
&\qquad\qquad\leq Ce^{C_1(T-t_0)} \gamma_0+Ce^{C_1(T-t_0)}\delta^{-1}(|x_0-y_0|^4+|z_0-y_0|^4)+Ce^{C_1(T-t_0)}\sqrt{h}\\
&\qquad\qquad\leq Ce^{C_1(T-t_0)} \gamma_0+Ce^{C_1(T-t_0)}\sqrt{h}, 
\end{align*}
where we used $\delta\geq\sqrt{h}$ in the second inequality.
Finally, this and \eqref{4.18} yield
\[
C_1(C_0e^{C_1(T-t_0)}\gamma_0+e^{C_1(T-t_0)}\sqrt{h})\leq Ce^{C_1(T-t_0)}\gamma_0+Ce^{C_1(T-t_0)}\sqrt{h},
\]
with $C>0$ depending only on $d,N$ and the regularity assumptions of $q,c,f$. 
Thus, if $C_1$ is sufficiently large depending only on the assumptions, we get a contradiction which finishes the proof of \eqref{4.17}, which finishes the proof. 
\end{proof}

\section{Generalization: a PDE perspective}
\label{sc5}

\quad In this section, we consider PI for HJB equations with a general Hamiltonian.
For convenient use of the Legendre transform, we write the system in the forward-in-time setting. 
It is easy to carry over to the backward-in-time setting.

\quad Suppose $\calH:[0,T]\times\bbR^d\times\bbR^d\to\bbR$ is continuous such that $\calH(t,x,p)$ is convex in $p$.
Let $\calL(t,x,\mu)$ be the Legendre transform of $\calH$, that is,
\[
\calL(t,x,\mu):=\sup_{p\in\bbR^d}\left[p\cdot \mu-\calH(t,x,p)\right] \qquad \text{ for } (t,x,\mu) \in [0,T] \times \bbR^d \times \bbR^d.
\]
We always have the following inequality
$\calL(t,x,\mu)+\calH(t,x,p)\geq p\cdot \mu$,
with equality holds if and only if $\mu=\nabla_p \calH(t,x,p)$, and if and only if $p=\nabla_\mu \calL(t,x,\mu)$. 

\quad The HJB equation is
\beq\lb{e.HJ}
\left\{
\begin{aligned}
    &\partial_t v(t,x)+\calH(t,x,\nabla v(t,x))=0 &&\quad\text{ in }(0,T)\times \bbR^d,\\
    &v(0,x)=q(x) &&\quad \text{ on }  \bbR^d.
\end{aligned}
\right.
\eeq
Under some assumptions (see \cite{barles1990regularity,tranbook}), it is a classical result that $v$ is uniformly Lipschitz continuous if $q$ is Lipschitz continuous. So we can assume 
\beq\lb{5.4}
\|\nabla v\|_{L^\infty([0,T]\times\bbR^d)}\leq M\quad\text{for some }M>0.
\eeq

\quad Now we take
$m_1:=\min_{\substack{|p|=2M,\\t\in [0,T],x\in\bbR^d}}\calH(t,x,p)$ and 
$m_2\geq \max_{\substack{|p|=3M,\\t\in [0,T],x\in\bbR^d}}[\calH(t,x,p)-m_1]/M$,
and we can assume that $m_2\geq 2$.
Then define
\[
\tilde{\calH}(t,x,p):=\left\{
\begin{aligned}
&\calH(t,x,p)   \qquad &&\text{ if }|p|\leq 2M,\\
&\max\left\{\calH(t,x,p), m_1+m_2(|p|-2M)\right\} \qquad &&\text{ if }
2M<|p|\leq 3M,\\
& m_1+m_2(|p|-2M) \qquad &&\text{ if }
|p|> 3M.
\end{aligned}
\right.
\]
It is not hard to verify that $\tilde{\calH}$ is continuous in all its dependencies, and is convex in $p$. 
Due to \eqref{5.4}, $v$ is also a solution of \eqref{e.HJ} with $\calH$ replaced by $\tilde{\calH}$. Moreover for $N:=m_2/2\geq 1$, we have
\beq\lb{5.3}
|\tilde{\calH}_p(t,x,p)|\leq 2N\quad\text{ in }[0,T]\times\bbR^d\times\bbR^d.
\eeq
We define $\tilde\calL$ as the Legendre transform of $\tilde\calH$. 
Since the goal is to approximate $v$, it suffices to study $\tilde\calH$ and $\tilde{\calL}$ instead of $\calH$ and ${\calL}$. From now on, with a slight abuse of notation, we write $\calH$ and $\calL$ as $\tilde\calH$ and $\tilde\calL$, respectively.

\quad With the modified operators, we can consider the semi-discretization. For $h>0$, 
\beq\lb{5.0}
\left\{
\begin{aligned}
    &\partial_t v^h(t,x)+\calH(t,x,\nabla^h v^h(t,x))=Nh\Delta^hv^h(t,x) &&\text{ in }(0,T)\times \bbR^d,\\
    &v^h(0,x)=q(x) &&\text{ on }  \bbR^d.
\end{aligned}
\right.
\eeq
As before, $N\geq \|\nabla_p\calH\|_\infty/2$ guarantees that the finite difference scheme is monotone. 
We also assume that there exists $C>0$ such that for all $t,x,p\in [0,T]\times\bbR^d\times\bbR^d$,
\beq\lb{5.2}
|\calH_t(t,x,p)|,\,|\calH_x(t,x,p)|\leq C(1+|p|),\quad |\calH(t,x,0)|\leq C.
\eeq
We can replace $C(1+|p|)$ by just $C$ for the modified operator.
We will not discuss the space-time discretization of \eqref{e.HJ} since it is similar.

\quad From the above discussions, we note the PDE in \eqref{e.HJ} can be rewritten as
\[
\partial_t v(t,x)+\sup_{\mu\in\bbR^d}\left\{\nabla v(t,x)\cdot\mu-\calL(t,x,\mu)\right\}=0
\]
and the supremum is achieved when 
$\mu(t,x)=\nabla_p\calH(t,x,\nabla v(t,x))$.
Therefore, we give the following iteration scheme for \eqref{5.0}. 
Fixing small $h>0$, we start with a uniformly bounded and Lipschitz continuous function $v^h_0(t,x)$, and then iteratively compute $v^h_n$ as follows.
For $n\geq 1$, let $v^h_n=v^h_n(t,x)$ be the solution to 
\beq\lb{5.1}
\left\{
\begin{aligned}
    &\partial_t v^h_n+\mu^h_{n-1}(t,x)\cdot \nabla^h v^h_n-\calL(t,x,\mu^h_{n-1}(t,x))=Nh\Delta^hv^h_n &&\text{ in }(0,T)\times \bbR^d,\\
    &v^h_n(0,x)=q(x) &&\text{ on }  \bbR^d
\end{aligned}
\right.
\eeq
where we denoted
$\mu^h_n(t,x):=\nabla_p \calH(t,x,\nabla^h v^h_{n}(t,x))$.
Note that $\calL(t,x,\mu^h_{n}(t,x))$ is finite due to $\mu_n^h\leq 2N$.
Essentially, $v_n^h$ solves a linearized equation of \eqref{5.0}. 

\quad Let $v_n^h$ (for each $n\geq 1$ with given $v_0^h$), $v^h$ and $v$ be, respectively, Lipschitz continuous solutions to \eqref{5.1}, \eqref{5.0} and \eqref{e.HJ}.
We have the following monotonicity property.

\begin{proposition}\lb{P.1}
Suppose $N\geq\max\{1,\|\nabla_p\calH\|_\infty/2\}$, and $\calH(t,x,p)$ is convex in $p$ and satisfies \eqref{5.3} and \eqref{5.2}. Let $q$ and $v_0^h$ be uniformly bounded and Lipschitz continuous for all $h>0$. Then the solutions $v_n^h$ are uniformly bounded for all $n\geq 1$ and $h>0$. Moreover, we have for all $n\geq 0$,
\[
v_{n+1}^h(t,x)\leq v_n^h(t,x) \qquad \text{ for all $(t,x)\in [0,T]\times\bbR^d$.}
\]
\end{proposition}

\quad 
In the current setting, $\mu_n^h(t,x),-\calL(t,x,\mu)$ and $\nabla_p\calH(t,x,p)$ correspond to $\alpha_n^h(t,x),c(t,x,a)$ and $\alpha(t,x,p)$ in Section \ref{sc2}, respectively.
The proof of Proposition \ref{P.1} then is identical to that of Proposition \ref{P.2.1} after converting the problem to a backward-in-time setting by considering $w^h_n(t,x)=v^h_n(T-t,x)$.

\quad Additionally, we have the following convergence results.  
The proof of Theorem \ref{T.5.2} follows those of Theorems \ref{T.2.4} and \ref{T.2.add}, and the proof of Theorem \ref{T.5.3} is analogous to those of Theorems \ref{T.2.8} and \ref{T.3.5}. 

\begin{theorem}\lb{T.5.2}
Under the assumptions of Proposition \ref{P.1}, 
for all $R\geq 1$ there exists a constant $C$ depending only on $T$ and the assumptions such that we have for all $t \in[ 0,T]$,
\[
\begin{aligned}
\int_{B_R}\left|v^h_n(t,x)-v^h(t,x)\right|^2\,dx \leq  C2^{-n}{h} e^{Ct/h}R^d,
\end{aligned}
\]
\[
\begin{aligned}
\int_{B_R}\left|\nabla_p\calH(t,x,\nabla^h v^h_n(t,x))- \nabla_p\calH(t,x,\nabla^h v^h(t,x))\right|^2\,dx 
\leq {C2^{-n}e^{Ct/h}}R^d/h.
\end{aligned}
\]
Moreover, we have
$\sup_{(t,x)\in [0,T]\times \bbR^d} |v^h(t,x)-v(t,x)|\leq C\sqrt{h}$.
\end{theorem}

\quad Next, let $\calH$ take the form
$\calH(t,x,p):=\sup_{a \in A}\left[c(t,x,a)+p\cdot f(t,x,a)\right]$,
where $A$ is some set, $c:[0,T]\times\bbR^d\times A\to\bbR$ and $f:[0,T]\times\bbR^d\times A\to\bbR^d$.

\begin{theorem}\lb{T.5.3}
Under the assumptions of Theorem \ref{T.5.2}, assume that $c(t,\cdot,a),f(t,\cdot,a)$ are bounded in $W^{2,\infty}(\bbR^d)$ uniformly for all $t\in[0,T]$ and $a \in A$.
Then for each $t\in [0,T]$, we have for a.e.  $x\in\bbR^d$,
\[
\alpha(t_h,x_h,\nabla^h v^h(t_h,x_h))\to \alpha(t,x,\nabla v(t,x))\quad\text{ as }h\to 0,
\]
where $[0,T]\times\bbR^d\ni(t_h,x_h)\to (t,x)$ as $h\to 0$.

\quad Moreover, there exists $C>0$ depending only on the assumptions such that for all $h\in (0,1)$, $t\in [0,T]$ and $x,y\in \bbR^d$,
$v^h(t,x+y)+v^{h}(t,x-y)-2v^{h}(t,x)\leq C\exp(CT)(|y|^2 +\sqrt{h})$.

\end{theorem}


\section{Numerical experiments}
\label{scnum}

\quad In this section, 
we provide numerical experiments to illustrate the convergence of PI. 
We take $f(t,x,a) = a$, $c(t,x,a) = \frac{1}{2} a^2$ and $q \equiv 0$,
so that $v_* \equiv 0$.
Figure \ref{fig} below (the semi-log plot) shows the exponential convergence of PI, 
corroborating Theorem \ref{T.4.2}.
For the vanishing viscosity approximations (Theorem \ref{T.4.3}),
we refer to \cite{QST24}  for numerical illustration.

\begin{figure}[!htb]
   \begin{minipage}{0.33\textwidth}
     \centering
       \includegraphics[width=\linewidth]{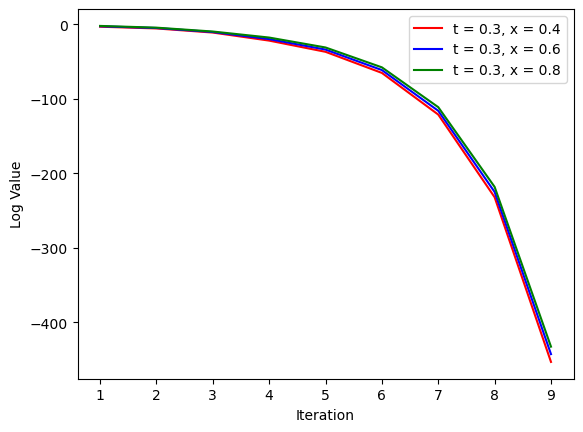}
   \end{minipage}\hfill
    \begin{minipage}{0.33\textwidth}
    \centering
     \includegraphics[width=\linewidth]{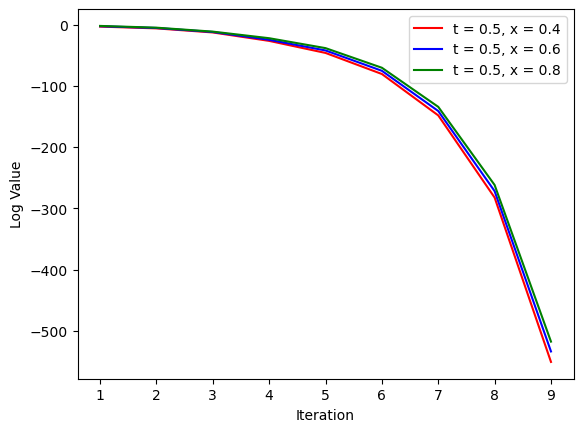}
   \end{minipage} \hfill
   \begin{minipage}{0.33\textwidth}
     \centering
       \includegraphics[width=\linewidth]{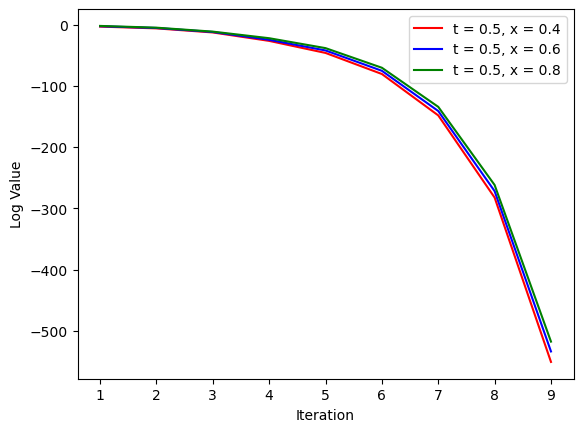}
   \end{minipage}
   \caption{Convergence of PI \eqref{4.1}--\eqref{4.1'} for $f(t,x,a) = a$, $c(t,x,a) = \frac{1}{2} a^2$ and $q \equiv 0$, with $A = [-2, 2]$, $\tau = 0.025$, $h = 0.1$ and $N = 2$.}
    \label{fig}
\end{figure}

\section{Conclusion}
\label{sc6}

\quad In this paper, we study the convergence rate of PI for optimal control problems in continuous time.
To overcome the problem of ill-posedness, 
we consider a semi-discrete scheme by adding a viscosity term using finite differences. 
We prove that PI for the semi-discrete scheme converges exponentially fast, 
and provide a bound on the discrepancy between the semi-discrete scheme and the optimal control.
We also study the discrete space-time scheme, where both space and time are discretized. 

\quad Several future directions are related to PI. 
First, we plan to relax condition (A2) by considering
\[
\alpha(t,x,p)\in \argmin_{a \in A} \left[ c(t,x,a) + p \cdot f(t,x,a)\right]
\]
to replace \eqref{eq:u}. 
One possible approach is to use the minimization property of $\alpha$ rather than its precise value in the proof. 
Note however that without the Lipschitz continuity of $\alpha$, we might not have the exponential convergence of the approximate optimal policies in Theorem \ref{T.2.add}.
Second, we have only proved a weak form of semi-concavity for $v^h$ in Theorem \ref{T.3.5}.
It remains an open question whether $v^h$ is semi-concave.

\quad Several future directions are related to the non-discretized PI. It remains unclear under which conditions on the model parameters the PI \eqref{eq:PDEiter}--\eqref{eq:uiter} is well-defined and converges exponentially fast. 
For instance, for $f(t,x,a) = a$, $c(t,x,a) = \frac{1}{2} |a|^2$ and $q \equiv 0$,
the HJB equation is 
$\partial_t v - \frac{1}{2} |\nabla v|^2 = 0$ and $v(T,x) = 0$,
which has the solution $v_* \equiv 0$.  
On the other hand, PI yields $v_n(t,x) = c_n(t) x^2$ with $c_1(t) = \frac{1}{2}$ for a suitable initialization. 
It is easy to check that $c_n(t) \le 2^{-n}$ for $n \ge 1$, and thus we get the exponential convergence of $v_n$ to $v_*$ on any compact set.
Moreover, it is also interesting to adapt PI to the differential game setting and design efficient numerical schemes (see e.g. \cite{HDDC20}). 
We refer to \cite{Kawecki-Sprekeler, Smears-Suli} for the use of PI to solve numerically fully nonlinear HJB and HJBI equations.

\bigskip
{\bf Acknowledgment}:
 We thank Elisabetta Carlini for pointing out some missing assumptions in Theorems \ref{T.2.8} and \ref{T.4.10}.
 
\bibliographystyle{abbrv}
\bibliography{unique}
\end{document}